\documentclass[reqno]{amsart}
\usepackage{amsmath,amsthm,amssymb,amsfonts,color}
\usepackage{hyperref,comment}
\usepackage{todonotes}
\usepackage{esint}
\usepackage[latin1]{inputenc}

\usepackage[shortlabels]{enumitem}

\renewcommand\div{\text{div}_v}

\newcommand\bR{\mathbb{R}}

\newcommand\bM{\mathbb{M}}

\newcommand\bS{\mathbb{S}}

\newcommand\cP{\mathcal{P}}

\newcommand\cR{\mathcal{R}}

\newcommand\tQ{\widetilde Q}

\newcommand\cM{\mathcal{M}}

\newcommand\cI{\mathcal{I}}

\newcommand\bH{\mathbb{H}}

\newcommand\sfv{\mathsf{v}}

\newcommand{\mysection}[1]{\section{#1}
 \setcounter{equation}{0}}
\numberwithin{equation}{section}

\numberwithin{equation}{section}

\newcommand\cbrk{\text{$]$\kern-.15em$]$}}
\newcommand\opar{
\text{\,\raise.2ex\hbox{${\scriptstyle |}$}\kern-.34em$($}}

\theoremstyle{plain}
\newtheorem{theorem}{Theorem}[section]
\newtheorem{lemma}[theorem]{Lemma}
\newtheorem{corollary}[theorem]{Corollary}
\newtheorem{proposition}[theorem]{Proposition}

\theoremstyle{definition}
\newtheorem{definition}[theorem]{Definition}

\newtheorem{assumption}[theorem]{Assumption}

\theoremstyle{remark}
\newtheorem{remark}[theorem]{Remark}

\newcommand{\nlimsup}{\operatornamewithlimits{\overline{lim}}}

\begin{document}

\title[$L_p$ estimates for kinetic Kolmogorov-Fokker-Planck equations]{Global $L_p$ estimates for kinetic Kolmogorov-Fokker-Planck equations in divergence form}

\author[H. Dong]{Hongjie Dong}
\address[H. Dong]{Division of Applied Mathematics, Brown University, 182 George Street, Providence, RI 02912, USA}
\email{Hongjie\_Dong@brown.edu }
\thanks{H. Dong was partially supported by the Simons Foundation, grant no. 709545, a Simons fellowship, grant no. 007638, the NSF under agreement DMS-2055244, and the Charles Simonyi Endowment at the Institute for Advanced Study.}

\author[T. Yastrzhembskiy]{Timur Yastrzhembskiy}
\address[T. Yastrzhembskiy]{Division of Applied Mathematics, Brown University, 182 George Street, Providence, RI 02912, USA}
\email{Timur\_Yastrzhembskiy@brown.edu}


\begin{abstract}
We  present  a priori estimates and unique solvability results in the mixed-norm Lebesgue spaces for kinetic Kolmogorov-Fokker-Planck (KFP) equation in  divergence form.
The leading coefficients are bounded  uniformly nondegenerate  with respect to the velocity variable $v$ and satisfy a  vanishing mean oscillation  (VMO) type condition. We consider the $L_2$ case separately and treat more general equations which include the relativistic KFP equation.
This work is a continuation of \cite{DY_21}, where the present authors studied the KFP equations in  nondivergence form.
\end{abstract}

\maketitle

\tableofcontents

\section{Introduction and the main results}

For any integer $d \ge 1$,  let  $\bR^d$ be a Euclidean space of points $(x_1, \ldots, x_d)$,
and for $T \in (-\infty, \infty]$,
 we set
$\bR^d_T = (-\infty, T) \times \bR^{d-1}$.
 Throughout the paper,
  $z$  is the triple $(t, x, v)$,
 where $t \in \bR$, and  $x, v \in \bR^d$.

The goal of this paper is to prove the a priori estimates and unique solvability results for the kinetic KFP equation
in   divergence form given by
\begin{equation}
                \label{1.1}
    \partial_t u - v \cdot D_x u - D_{v_i} (a^{i j} (z) D_{v_j} u) + \div (\overline{b} u)
    + b \cdot D_{v} u + c u + \lambda u
 = \div \vec f + g.
\end{equation}

\subsection{Notation and assumptions}
For $x_0\in \bR^d$, $z_0\in \bR^{1+2d}$, and $r,R>0$, we introduce
\begin{align*}
  &  	B_r (x_0) = \{\xi \in \bR^d: |\xi-x_0| < r\},\\
&
	Q_{r, R} (z_0)
	=  \{z: -r^2<t - t_0<0, |v-v_0| < r, |x - x_0 + (t - t_0) v_0|^{1/3} < R\},\\
 &
	\tQ_{r, R} (z_0)
	=  \{z: |t - t_0|^{1/2} < r, |v-v_0| < r, |x - x_0 + (t - t_0) v_0|^{1/3} < R\},\\
 &
        Q_r (z_0) = Q_{r, r} (z_0), \quad \tQ_r (z_0) = \tQ_{r, r} (z_0) \quad Q_r  = Q_r (0), \quad \tQ_r = \tQ_r (0).
  \end{align*}
For $f \in L_{1, \text{loc}} (\bR^d)$ and a Lebesgue measurable set, we denote by $(f)_A$ or $\fint_{A} f \, dx$ the average of $f$ over $A$.
Furthermore, for $c > 0$, $T \in (-\infty, \infty]$, and $f \in L_{1, \text{loc}} (\bR^{1+2d}_T)$,  we introduce variants of maximal and sharp functions
\begin{equation}
			\label{eq3.38}
 \begin{aligned}
    &  	\bM_{c, T} f (z_0)
	= \sup_{r > 0, z_1 \in  \overline{\bR^{1+2d}_T}: z_0 \in  Q_{r, c r} (z_1) } \fint_{Q_{r, c r} (z_1) } |f (z)| \, dz,
	\quad
		\cM_T f  := \bM_{1, T} f,\\
	&f^{\#}_{ T} (z_0) = \sup_{r > 0, z_1 \in \overline{\bR^{1+2d}_T}: z_0 \in  Q_r (z_1) } \fint_{Q_{r} (z_1) } |f (z) -  (f)_{Q_{r} (z_1) }| \, dz.
 \end{aligned}
\end{equation}

We impose the following assumptions on the coefficients.
\begin{assumption}

                \label{assumption 2.1}
The coefficients $a (z) = (a^{ij} (z), i, j = 1, \ldots, d)$ are bounded measurable functions  such that  for some $\delta \in (0, 1)$,
$$
    \delta |\xi|^2 \le a^{i j} (z) \xi_i \xi_j,\quad |a^{ij}(z)|\le \delta^{-1}, \quad \forall \xi \in \bR^d,\, z \in \bR^{1+2d}.
$$
\end{assumption}

The next assumption can be viewed as a kinetic $VMO_{x, v}$ assumption with respect to
\begin{equation}
                \label{1.2}
        \rho (z, z_0) =   \max\{ |t-t_0|^{1/2},  |x - x_0 + (t - t_0) v_0|^{1/3}, |v-v_0|\},
\end{equation}
which satisfies  all the  properties of the quasi-metric except the symmetry.
It is analogous to the $VMO_x$ condition from the theory of nondegenerate parabolic equations with rough coefficients (see Chapter 6 of \cite{Kr_08}).
\begin{assumption}$(\gamma_0)$
                  \label{assumption 2.2}
There exists $R_0 \in (0,1)$ such that for any $z_0$ and $r \in (0, R_0]$,
$$
     \text{osc}_{x, v} (a, Q_r (z_0)) \le \gamma_0,
$$
where
\begin{align*}
  & \text{osc}_{x, v} (a, Q_r (z_0))\\
&= \fint_{(t_0 - r^2, t_0)} \fint_{D_r (z_0, t) \times D_r (z_0, t)}
   |a (t, x_1, v_1) - a (t, x_2, v_2)| \, dx_1dv_1 dx_2dv_2 \, dt,
\end{align*}
  and
  $$
    D_r (z_0, t) =  \{(x, v): |x  - x_0 + (t-t_0) v_0|^{1/3} < r, |v-v_0| < r  \}.
$$
\end{assumption}

\begin{remark}
In this remark, we give examples when
Assumption \ref{assumption 2.2} is satisfied.
Throughout the remark,  $\omega: [0, \infty) \to [0, \infty)$
 is an increasing function
such that $\omega (0+) = 0$.

\textit{Anisotropic $VMO_{x, v}$ condition:}
\begin{equation}
			\label{eq2.3.1}
\begin{aligned}
	&
\text{osc}_{x, v}' (a, r)    := \sup_{t, x, v} r^{-8d}\\
& \times \int_{  x_1, x_2 \in B_{r^3} (x) } \int_{v_1, v_2 \in B_r (v) } |a (t, x_1, v_1) - a (t, x_2, v_2)| \, dx_1dx_2\, dv_1dv_2 \le \omega (r).
\end{aligned}
\end{equation}
Since $\text{osc}_{x, v} (a, Q_r (z_0))    \le \text{osc}_{x, v}' (a, r)$, if anisotropic $VMO_{x, v}$ condition holds, then for any $\gamma_0 \in (0, 1)$,
   Assumption       \ref{assumption 2.2} $(\gamma_0)$ holds.

\textit{Continuity with respect to the anisotropic distance} $\text{dist} ((x, v), (x', v')) := |x-x'|^{1/3} +|v-v'|$:
for  any $t,  x, x', v, v'$,
$$
	|a (t, x,v) - a (t, x',v')| \le \omega \big(\text{dist} ((x, v), (x', v'))\big).
$$
Note that if this condition holds, then \eqref{eq2.3.1} is true, and, therefore, for any $\gamma_0 \in (0, 1)$,   Assumption \ref{assumption 2.2} $(\gamma_0)$ is satisfied.
\end{remark}

\begin{assumption}
                  \label{assumption 2.3}
The functions $b  = (b^i, i = 1, \ldots, d)$, $\overline{ b} =  (\overline{ b}^i, i = 1, \ldots, d)$, and $c$ are bounded measurable  on $\bR^{1+2d}$,
 and they satisfy  the condition
$$
	|b| + |\overline{ b}| + |c| \le L
$$
for some constant $L > 0$.
\end{assumption}

\subsection{Function spaces.}
Below we define the mixed-norm Lebesgue and Sobolev spaces.
In all these definitions, $G \subset \bR^{1+2d}$ is an open set, and $p, r_1, \ldots, r_d, q > 1$ are numbers.
\begin{definition}
We say that $w$ is a weight  on $\bR^d$
if $w$ is a locally integrable function that is
positive almost everywhere. Let $w_i, i = 0, 1, \ldots, d,$ be weights on $\bR$.
By $L_{p,  r_1, \ldots, r_d, q}(G, w)$ with
\begin{equation}
			\label{eq2.2}
w=w(t,v)=w_0(t) w_1(v_1)\cdots w_d(v_d),
\end{equation}
we denote the space of all Lebesgue measurable functions on $\bR^{1+2d}$ such that
\begin{align*}
   &\|f\|_{ L_{p,  r_1, \ldots, r_d, q} ( G, w)}\\
   & =
    \big|\int_{\bR}
    \big| \ldots \big|\int_{\bR} \big|\int_{\bR^d} |f|^p (z) 1_{G}(z) \,  dx\big|^{\frac {r_1} p} \, w_1(v_1)  dv_1\big|^{\frac {r_2}{r_1}} \ldots w_d (v_d) dv_d\big|^{\frac q {r_d}}\, w_0(t) dt
    \big|^{\frac 1 q},
\end{align*}
and for $\alpha \in (-1, p-1)$, we set $L_{p; r_1, \ldots, r_d} (\bR^{1+2d}_T,  |x|^{\alpha} \prod_{i = 1}^d w_i (v_i))$
to be the weighted mixed-norm Lebesgue space with the norm
\begin{equation}
\begin{aligned}
			\label{eq2.2.1}
&\|f\|_{ L_{p;  r_1, \ldots, r_d} (G, |x|^{\alpha} \prod_{i = 1}^d w_i (v_i) )} \\
&
=   \big|\int_{\bR} \ldots \big|\int_{\bR} \big|\int_{\bR^{d+1}} |f|^p (z) 1_{G}(z) |x|^{\alpha} \,  dxdt\big|^{\frac {r_1} p} \, w_1(v_1)  dv_1\big|^{\frac {r_2}{r_1}}
\ldots w_d (v_d) dv_d\big|^{\frac{1}{r_d}}.
\end{aligned}
\end{equation}
Furthermore, for a vector-valued function $\vec f  = (f_1, \ldots, f_d)$, we write
$$
\vec f \in L_{p, r_1, \ldots, r_d, q} (G, w)  \,  \bigg(\text{or } L_{p;  r_1, \ldots, r_d} (G, |x|^{\alpha} \prod_{i = 1}^d w_i (v_i))\bigg)
$$
if each component $f_i$ is in $L_{p, r_1, \ldots, r_d, q} (G, w)$ $\big(\text{or }L_{p;  r_1, \ldots, r_d} (G, |x|^{\alpha} \prod_{i = 1}^d w_i (v_i))\big)$.
\end{definition}

Throughout this  paper, $w  =  w (t, v)$ is a weight on $\bR^{1+d}$.

\begin{definition}
By $\bH^{-1}_{p, r_1, \ldots, r_d, q} (G, w)$ we denote the set of all functions $u$ on $G$ such that there exist $\vec f, g \in L_{p, r_1, \ldots, r_d, q} (G, w)$ satisfying
\begin{equation}
			\label{eq2.4}
    u = \div \vec f + g.
\end{equation}
The norm is given by
$$
   \|u\|_{  \bH^{-1}_{p, r_1, \ldots, r_d, q} (G, w)  }  = \inf  \big(\|\vec f\|_{L_{p, r_1, \ldots, r_d, q} (G, w)} + \|g\|_{L_{p, r_1, \ldots, r_d, q} (G, w)}\big),
$$
where the infimum is taken over all $\vec f$ and $g$ satisfying \eqref{eq2.4}.
\end{definition}

Here is the definition of the kinetic (ultraparabolic) Sobolev spaces.
The first one is designed to treat the divergence form equations, whereas the second one is used to work with the KFP equations in nondivergence form.

\begin{definition}
By $\mathbb{S}_{p, r_1, \ldots, r_d, q} (G, w)$
we denote  the Banach space
of all functions $u$
such that $u,  D_v u \in L_{p, r_1, \ldots, r_d, q} (G, w)$, and
$(\partial_t  -  v \cdot  D_x)u \in \mathbb{H}^{-1}_{p, r_1, \ldots, r_d, q} (G, w)$. The norm is defined as follows:
$$
    \|u\|_{\mathbb{S}_{p, r_1, \ldots, r_d, q} (G, w)} = \|u\| + \| D_v u\|
    + \|\partial_t u - v \cdot  D_x u\|_{ \mathbb{H}^{-1}_{p, r_1, \ldots, r_d, q} (G, w)},
$$
where $\|\cdot\| = \|\cdot\|_{ L_{p, r_1, \ldots, r_d, q} (G, w)}$.
\end{definition}

\begin{definition}
Let $S_{p, r_1, \ldots, r_d, q} (G, w)$
be the  Banach space
of  functions $u$
such that $u$,  $D_v u$, $D^2_v u$, $(\partial_t  -  v \cdot D_x)u   \in L_{p, r_1, \ldots, r_d, q} (G, w)$, and  the norm is given by
$$
    \|u\|_{ S_{p, r_1, \ldots, r_d, q} (G, w)} = \||u| +  |D_v u| + |D^2_v u| + |\partial_t u - v \cdot D_x u|\|_{ L_{p, r_1, \ldots, r_d, q} (G, w)}.
$$
\end{definition}

If $w  \equiv 1$ or $p = q = r_1 = r_2 = \ldots =  r_d$,
we drop $w$ or $q, r_1, \ldots, r_d$ from the above notation.

We define the spaces
\begin{align*}
&
\bH^{-1}_{p;  r_1, \ldots, r_d} (G, |x|^{\alpha} \prod_{i = 1}^d w_i (v_i)),  \quad \mathbb{S}_{p;  r_1, \ldots, r_d} (G, |x|^{\alpha} \prod_{i = 1}^d w_i (v_i)),\\
& \text{and} \, \, S_{p;  r_1, \ldots, r_d} (G, |x|^{\alpha} \prod_{i = 1}^d w_i (v_i))
\end{align*}
  in the same way.

By $\mathcal{S} (\bR^d)$ we denote the set of Schwartz functions, and by $C^{\infty}_0 (\bR^{d})$ - the set of all smooth compactly supported functions on $\bR^d$.
\begin{definition}
            \label{definition 1.1}
We write $u \in C_0 (\bR^d)$ if $u$ is a continuous function vanishing at infinity.
For $k \in \{1, 2, \ldots\},$ by $C^k_0 (\bR^d)$, we mean the subspace of $C_0 (\bR^d)$ of functions such that $D^j u \in C_0 (\bR^d), j = 1, \ldots, k$.
\end{definition}

\subsection{Main results}
				
\subsubsection{$L_p$ theory for KFP equations with  VMO coefficients}
Denote
\begin{equation}
                \label{1.1.0}
    \cP = \partial_t  - v \cdot D_x + D_{v_i} (a^{i j} D_{v_j}).
\end{equation}

\begin{definition}
                \label{definition 2.1}
For  $T \in  (-\infty,  \infty]$, we  say   that $u  \in \bS_{p, r_1, \ldots, r_d, q} (\bR^{1+2d}_T, w)$ is a solution to Eq. \eqref{1.1}
if the identity   \eqref{1.1}  holds in the space $\bH^{-1}_{p, r_1, \ldots, r_d, q} (\bR^{1+2d}_T, w)$.
Furthermore, for $-\infty<S<T\le \infty$,
$$
  u \in \bS_{p, r_1, \ldots, r_d, q} ((S, T) \times \bR^{2d}, w)
$$
is a solution to the Cauchy problem
\begin{equation}
                \label{2.1}
     \cP u + \div (\overline{b}  u) + b^i D_{v_i} u + c  u = \div \vec f + g, \quad u (S, \cdot) = 0
\end{equation}
if there exists $\widetilde u \in \bS_{p, r_1, \ldots, r_d, q} (\bR^{1+2d}_T, w)$ such that $\widetilde u = u$
on $(S, T) \times \bR^{2d}$,
$\widetilde u = 0$ on $(-\infty, S)  \times \bR^{2d}$, and the equality
$$
    \cP  u + \div (\overline{ b} u)  + b^i D_{v_i} u + c u  = \div \vec f + g
$$
holds in $\bH^{-1}_{p, r_1, \ldots, r_d, q} ((S, T) \times \bR^{2d}, w)$.
Similarly, we define a solution in the space $\bS_{p; r_1, \ldots, r_d} ((S, T) \times \bR^{2d}, |x|^{\alpha} \prod_{i = 1}^d w_i (v_i))$.
\end{definition}

\begin{definition}[$A_p$-weight]
For a number $p > 1$, we write $w \in A_p (\bR^d)$
if $w$ is a weight on $\bR^d$
such that
\begin{equation}
\begin{aligned}
		\label{eq1.12}
    [w]_{A_p (\bR^d)}  : = & \sup_{ x_0 \in \bR^d, r > 0}
    \Big(\fint_{    B_r (x_0) } w (x) \, dx\Big) \\
&\times \Big(\fint_{ B_r (x_0) } w^{-1/(p-1)}(x) \, dx\Big)^{p-1} < \infty.
\end{aligned}
\end{equation}
\end{definition}

\begin{remark}
                \label{rem2.1}
An example of an $A_p (\bR^d)$-weight is $w (x) = |x|^{\alpha}, \alpha \in (-d, d(p-1))$ (see, for instance,   Example 7.1.7 in \cite{G_14}).
\end{remark}

\begin{definition}
For $s \in \bR$, the fractional Laplacian $(-\Delta_x)^s$ is defined as a Fourier multiplier with the symbol $|\xi|^{2s}$.
Furthermore, when $s \in (0, 1/2)$,  for any Lipschitz function $u \in \cup_{p \in [1, \infty]} L_p (\bR^d)$,
the following pointwise formula is valid:
\begin{equation}
            \label{eq1.20}
	(-\Delta_x)^s u (x) = c_{d, s } \int_{\bR^d} \frac{u (x) - u (x+y)}{|y|^{d+2s}} \, dy,
\end{equation}
where $c_{d, s}$ is a constant depending only on $d$ and $s$. When $s \in [1/2, 1)$ and $u$ is bounded and $C^{1,1}$, the formula still holds provided that the integral is understood as the principal value.
For  $s \in (0, 1)$ and $u \in L_p (\bR^d)$,  $(-\Delta_x)^s u$ is understood  as a  distribution  given by
\begin{equation}
			\label{eq1.16}
	((-\Delta_x)^{s} u, \phi) = (u, (-\Delta_x)^{s} \phi), \quad \phi \in C^{\infty}_0 (\bR^d).
\end{equation}
To prove that \eqref{eq1.16} defines a distribution, one needs to use the fact that
\begin{equation}
			\label{eq1.15}
	|(-\Delta_x)^{s} \phi (z)|  \le N (d, \phi) (1+|x|)^{-d - 2s}, \quad \phi \in C^{\infty}_0 (\bR^d).
\end{equation}
Furthermore, by \eqref{eq1.15}, for any $\alpha \in (-d - 2 s p, d(p-1))$,
$$
	(-\Delta_x)^{s} \phi \in L_{p/(p-1)}(\bR^d, |x|^{-\alpha/(p-1)}),
$$
 so that \eqref{eq1.16} defines the distribution $(-\Delta_x)^s u$ for any $u \in L_p (\bR^d, |x|^{\alpha})$.
For a detailed discussion of the fractional Laplacians, we refer the reader to \cite{S_19}.
\end{definition}

\noindent\textbf{Convention.}
By $N = N (\cdots)$, we denote a  constant depending only on the parameters inside the
parenthesis. A constant $N$ might change from line to line. Sometimes, when it is clear what parameters $N$  depends on, we omit  them.

\begin{theorem}
            \label{theorem 1.1}
 Let
\begin{itemize}[--]
\item $p$, $r_1, \ldots, r_d$, $q > 1$, $K \ge 1$ be numbers,
$T \in (-\infty, \infty]$,
\item $w_i, i = 0, \ldots, d,$ be weights on $\bR$ such that
\begin{equation}
			\label{eq1.0}
    [w_0]_{A_q(\bR)}, [w_i]_{ A_{r_i} (\bR) } \le K,\quad i = 1, \ldots, d,
\end{equation}
\item $w$ be defined by  \eqref{eq2.2},
\item Assumptions \ref{assumption 2.1} and \ref{assumption 2.3}
hold.
\end{itemize}

There exists a constant
$$
\gamma_0   = \gamma_0  (d, \delta, p, r_1, \ldots, r_d, q, K)  > 0
$$
such that if Assumption \ref{assumption 2.2} $(\gamma_0)$ holds,
then, the following assertions are valid.

$(i)$ There exists a constant
$$
\lambda_0 =   \lambda_0 (d, \delta, p, r_1, \ldots, r_d, q, K, L, R_0)  > 1
$$
such that
for any $\lambda \ge \lambda_0$, $u \in \bS_{p, r_1, \ldots, r_d, q} (\bR^{1+2d}_T, w)$,
 and
 $g, \vec f \in L_{p, r_1, \ldots, r_d, q} (\bR^{1+2d}_T, w)$
 satisfying Eq. \eqref{1.1},
one has
\begin{equation}
			\label{2.1.2}
 \lambda^{ 1/2 }  \|u\| + \|D_v u\| +  \|(-\Delta_x)^{1/6} u\|
    \le N  \lambda^{-1/2} \|g\| + N \|\vec f\|,
\end{equation}
where $R_0 \in (0, 1)$ is the constant in  Assumption \ref{assumption 2.2} $(\gamma_0)$,
 $$
 \|\,\cdot\,\| = \|\,\cdot\,\|_{ L_{p, r_1, \ldots, r_d, q} (\bR^{1+2d}_T, w)  },
\quad\text{and}\quad N = N (d, \delta, p, r_1, \ldots, r_d, q, K).
$$

$(ii)$ For any $\lambda \ge \lambda_0$,  $\vec f, g \in L_{p, r_1, \ldots, r_d, q} (\bR^{1+2d}_T, w)$,
Eq. \eqref{1.1}
has a unique solution $u \in \bS_{p, r_1, \ldots, r_d, q} (\bR^{1+2d}_T, w)$.
Here $\lambda_0$ is the constant from the assertion $(i)$.

$(iii)$ For any numbers $-\infty<S < T<\infty$
and
$\vec f, g \in L_{p, r_1, \ldots, r_d, q} ((S, T) \times \bR^{2d}, w)$,
Eq. \eqref{2.1} has a unique solution
$u  \in \bS_{ p, r_1, \ldots, r_d, q } ((S, T) \times \bR^{2d}, w)$.
In addition,
$$
 	  \|u\| + \|D_v u\|  +  \|(-\Delta_x)^{1/6} u\|
   \le N \|\vec f\| + N \|g\|,
$$
where
$$
    \|\,\cdot\,\|=\|\,\cdot\,\|_{L_{p, r_1, \ldots, r_d, q} ((S, T) \times \bR^{2d}, w)  } \, \,  \text{and} \, \,  N = N (d, \delta, p,r_1, \ldots, r_d, q, K, L, R_0, T-S).
$$

$(iv)$  Let $\alpha \in (-1, p-1)$.
The assertions $(i) - (iii)$ also hold in the case when
\begin{align*}
	&\vec f, g  \in L_{p; r_1, \ldots, r_d} (\bR^{1+2d}_T, |x|^{\alpha} \prod_{i = 1}^d w_i (v_i)),  \\
& u \in \bS_{p; r_1, \ldots, r_d} (\bR^{1+2d}_T, |x|^{\alpha} \prod_{i = 1}^d w_i (v_i)).
\end{align*}
Furthermore,   one needs to take into account the dependence of constants  $\gamma_0, \lambda_0, N$ on $\alpha$ and remove the dependence on $q$.
\end{theorem}

\begin{remark}
			\label{remark 2.2}
The assertion $(iii)$ is derived from $(ii)$ by using an exponential multiplier (see, for example, Theorem  2.5.3 of \cite{Kr_08}).
\end{remark}

\begin{remark}
By viewing an elliptic equation as a steady state parabolic equation,
we can obtain the corresponding results for elliptic equations when the coefficients and data are independent of the temporal variable. See, for example, the proof of \cite[Theorem 2.6]{Kr07}.
\end{remark}

To the best of the present authors' knowledge, Theorem \ref{theorem 1.1} provides the first global a priori $L_p$-estimate with $p \ne 2$ for kinetic KFP equations in divergence form with non-smooth coefficients (see Section \ref{subsection 1.4}).
We also prove the first unique solvability result in $\bS_p$ space for Eq. \eqref{1.1}  in the case of the variable coefficients $a^{i j}$.
To the best of our knowledge, the imposed assumption on the leading coefficients $a^{i j}$ (see Assumption \ref{assumption 2.2}) is weaker than the ones in the existing literature (see Section \ref{subsection 1.4}).

To prove Theorem \ref{theorem 1.1}, we use the results and techniques of \cite{DY_21}, which are based on N.V. Krylov's kernel free approach to nondegenerate parabolic equations (see \cite{Kr_08}, Chapters 4 - 7).
 The main part of the argument is  the mean oscillation estimates  of $(-\Delta_x)^{1/6} u$,  $\lambda^{1/2} u$, and $D_v u$ in the case when the coefficients $a^{i j}$ are independent of the $x$ and $v$ variables. Our proof of these inequalities does not involve the fundamental solution of the KFP operator. Instead, we use the scaling properties of the KFP equation combined with localized $L_p$-estimates and a pointwise formula for  fractional Laplacians in order to get a mean oscillation estimates of solutions. By using the method of frozen coefficients, we generalize the aforementioned mean oscillation estimates to the case when $a^{ij}$ also depend on $x$ and $v$. Once such inequalities are established, the a priori estimates are obtained by using the Hardy-Littlewood and Fefferman-Stein type theorems (see Theorem \ref{theorem A.4}).

\subsection{\texorpdfstring{$L_2$}{} theory for the  kinetic equations with bounded measurable coefficients}
Let $x \in \bR^d$ and $\sfv \in \bR^{d_1}$ for some $d_1 = \{1, 2, \ldots\}$,
and $\alpha$ be a mapping from $\bR^{d_1}$ to $\bR^d$, and $a, b, \bar b, c, \vec f, g$ be functions of $t, x, \sfv$.

We consider the  equation
\begin{equation}
			\label{1.3.2}
	\cP_{\alpha}  u + \text{div}_{\sfv} (\bar b u) + b \cdot D_{\sfv} u + (c+\lambda) u = \text{div}_{\sfv} \vec f + g,
\end{equation}
where
$$
	\cP_{\alpha} u = \partial_t u  + \alpha (\sfv) \cdot D_x u - D_{\sfv_i} (a^{ i j} D_{\sfv_j} u).
$$

\begin{assumption}
			\label{assumption 1.3.2}
The function $\alpha$ is  such that for some $\theta \in (0, 1]$,
$$
	\sup_{\sfv \ne \sfv'} \frac{|\alpha (\sfv) - \alpha (\sfv')|}{|\sfv-\sfv'|^{\theta}} < \infty.
$$
\end{assumption}

\begin{remark}
Here we give examples of the equations of type \eqref{1.3.2} appeared in the existing literature.

\textit{Kinetic  equations:}  $d_1 = d$ and $\alpha (\sfv) = \pm \sfv$ or $\alpha = \frac{\pm \sfv}{ (1+|\sfv|^2)^{1/2} }$.
In the second case, Eq \eqref{1.3.2} with such $\alpha$ can be viewed as a relativistic counterpart of Eq. \eqref{1.1}.

\textit{The Mumford equation.} Another example comes from computer vision. In \cite{M_94}, D. Mumford considered the operator
$$
	\partial_t u + \cos (\sfv) D_{x_1} u + \sin (\sfv) D_{x_2} u - D^2_{\sfv} u, \, \, t, \sfv, x_1, x_2 \in \bR,
$$
which is an operator of the Kolmogorov-Fokker-Planck type.
For the discussion of certain PDE aspects of this operator, see \cite{KPP_16}.
\end{remark}

\begin{definition}
For $\alpha: \bR^{d_1} \to \bR^d$ and an open set $G \subset \bR^{1+d+d_1}$,
 we say that $u \in \mathfrak{S}_2 (G, \alpha)$ if $u, D_\sfv u \in  L_2 (G)$ , and $\partial_t u + \alpha (\sfv) \cdot D_x u \in \bH^{-1}_2 (G)$.
\end{definition}

Here is the $\mathfrak{S}_2 (\bR^{1+d+d_1}_T, \alpha)$ unique solvability result for Eq. \eqref{1.3.2}.
\begin{theorem}
			\label{theorem 1.10}
Let $T \in (-\infty, \infty]$, and $a$, $b$, $\bar b$ be functions satisfying Assumptions \ref{assumption 2.1} and \ref{assumption 2.3} (with $\bR^{1+2d}$ replaced with $\bR^{1+d+d_1}$),
and $\alpha: \bR^{d_1} \to \bR^d$ satisfy Assumption \ref{assumption 1.3.2}.
Then, there exists $\lambda_0 = \lambda_0 (d, d_1, \delta, L) > 0$ such that for any $\lambda \ge \lambda_0$,
and functions $u \in \mathfrak{S}_2 (\bR^{1+d+d_1}_T, \alpha)$, $\vec f, g \in L_2 (\bR^{1+d+d_1}_T)$ satisfying \eqref{1.3.2}, we have
\begin{equation}
			\label{eq1.10.1}
	\lambda^{1/2} \|u\| + \|D_{\sfv} u\|  \le N \|\vec f\| +  N   \lambda^{-1/2} \|g\|,
\end{equation}
where
$$
	N  = N (d, d_1, \delta), \quad	\| \cdot \| = \| \cdot \|_{  L_2 (\bR^{1+d+d_1}_T) }.
$$
In addition, for any $\vec f, g \in L_2 (\bR^{1+d+d_1}_T)$ and $\lambda \ge \lambda_0$, Eq. \eqref{1.3.2} has a unique solution $u \in  \mathfrak{S}_{2} (\bR^{1+d+d_1}_T, \alpha)$.

$(ii)$ For any numbers $S < T$ and $\vec f, g \in L_2 ((S, T) \times \bR^{d+d_1})$, the Cauchy problem
$$
	\cP_{\alpha} u + \text{div}_{\, \sfv} (\bar b u) +  b \cdot D_{\sfv} u + c u = \text{div}_{\, \sfv} \vec f + g, \quad u (S, \cdot) = 0
$$
has a unique solution $u \in  \mathfrak{S}_2 ((S, T) \times \bR^{d+d_1}, \alpha)$
and, furthermore,
$$
		 \|u\| + \|D_v u\|  \le N  \|\vec f\| +  N \|g\|,
$$
where
$$
	N  = N (d, d_1, \delta, L, T-S), \quad \| \cdot \| = \| \cdot \|_{  L_2 ((S, T) \times  \bR^{d+d_1}) }.
$$
\end{theorem}

\begin{corollary}
			\label{corollary 1.11}
In the case when $d_1 = d$ and $\alpha  = - \sfv$, which corresponds to the kinetic KFP equation, in addition to \eqref{eq1.10.1},  we have
$$
	  \|(-\Delta_x)^{1/6} u\|_{  L_2 (\bR^{1+2d}_T) }
 \le N   \|\vec f\|_{  L_2 (\bR^{1+2d}_T) } +   N \lambda^{-1/2}  \|g\|_{  L_2 (\bR^{1+2d}_T) },
$$
where $N  = N (d,  \delta)$.
\end{corollary}
\begin{proof}
 Note that  the identity
\begin{align*}
	&\partial_t u - \sfv \cdot D_x u - \Delta_{\sfv} u +  \text{div}_{\sfv} (\bar b u) +  b \cdot D_{\sfv} u + (c + \lambda) u \notag\\
&= D_{\sfv_i} ((a^{i j} - \delta_{i j}) D_{\sfv_j} u)+\text{div}_{\sfv} \vec f + g
\end{align*}
is true. Here $\Delta_{\sfv}$ is the Laplacian in the $\sfv$ variable.
Then, by Theorem \ref{theorem 1.1},
\begin{align*}
&	\|(-\Delta_x)^{1/6} u\|_{  L_2 (\bR^{1+2d}_T) }\\
& \le N \|(a^{i j} - \delta_{i j}) D_{\sfv_j}u\|_{  L_2 (\bR^{1+2d}_T) }
+N  \|\vec f\|_{  L_2 (\bR^{1+2d}_T) } +   N \lambda^{-1/2} \|g\|_{  L_2 (\bR^{1+2d}_T) }.
\end{align*}
This combined with Assumption \ref{assumption 2.1} and  \eqref{eq1.10.1} give the desired estimate.
\end{proof}

The results of Theorem \ref{theorem 1.10} and Corollary \ref{corollary 1.11} are not surprising,
however, the present authors have not seen such  assertions in their full  generality in the existing literature.

\subsection{Motivations}
\textit{Filtering.} Stochastic partial differential equation (SPDEs) in divergence form appear naturally in the theory of partially observable diffusion processes.
 In particular, the unnormalized conditional probability density of the unobservable component of the diffusion process with respect to the observable one satisfies a linear SPDE called Dunkan-Mortensen-Zakai equation (see, for example, \cite{RL_18}).
 In \cite{Kr_10}, N.V. Krylov showed that  the $L_p$ theory of SPDEs can be used to deduce certain regularity properties of the unnormalized density.
For the Langevin type diffusion processes, such a program was carried out  in \cite{ZZ_21} (see also \cite{PP_19}). In particular, the authors developed the Besov regularity theory for the equation
$$
	du = [v \cdot D_x u - a^{i j} D_{v_i v_j} u + b \cdot D_v u + f] \, dt  + [\sigma^k \cdot D_v u + g^k]\, dw^k,
$$
where $w_{k}, k \ge 1$ is a sequence of independent standard Wiener processes.
In the same paper, they  used that regularity theory to  show that the unnormalized conditional probability density is a continuous function.
We believe that Theorem \ref{theorem 1.1} is useful in developing the theory of stochastically forced KFP equations in \textit{divergence form} in the case when the forcing term $g$ belongs to the $L_p$ space with respect to the probability measure and  $t, x, v$.

\textit{Kinetic theory.}
The nonlinear Landau equation is an important model of collisional plasma which has been studied extensively (see, for example, \cite{AV_04},  \cite{M_18}, \cite{GIMV_19}, \cite{KGH}). A linearized version of this equation has the form of Eq. \eqref{1.1}.
Recently, there has been an interest in developing the $L_p$ theory of KFP equations with rough coefficients. Such results are useful for establishing the well-posedness of diffusive kinetic equations in bounded domains with the specular reflection boundary condition (see \cite{DGY_21}), and for the conditional regularity problem (see \cite{GIMV_19}).

\subsection{Related works}
			\label{subsection 1.4}
\textit{Divergence form equations}

Many  articles on KFP equations in divergence form are concerned with the local boundedness, Harnack inequality (including a nonhomogeneous version), and H\"older continuity of solutions to \eqref{1.1} (see \cite{PP_04}, \cite{WZ_09}, \cite{GIMV_19}, \cite{M_18}, \cite{GI_21}, \cite{GM_21}, \cite{Z_21}). See also the references in \cite{AP_20}.

It seems that there are few works on the Sobolev space theory of Eq. \eqref{1.1}.
Previously, an interior $L_p$ estimate of $D_v u$ for \eqref{1.1} was established in \cite{MP_98} under the assumption that $u, \vec f \in L_{p, \text{loc}}$, $u,D_v u, (\partial_t - v \cdot D_x) u \in L_{2, \text{loc}}$  and $g \equiv 0$
by using the explicit representation of the fundamental solution of the operator $\cP$ (see \eqref{1.1.0}) and singular-integral techniques. In addition, in the same work, it  was showed that if $p$ is large enough, then $u$ is locally H\"older continuous with respect to $\rho$ (see \eqref{1.2}).
The authors imposed the VMO condition with respect to $\rho$ on the coefficients $a^{i j}$, which is stronger than Assumption \ref{assumption 2.2}.
It can be seen from \eqref{1.2} that such an assumption might not be satisfied even when the coefficients  $a^{i j} = a^{i j} (x, v)$ are smooth, bounded, and independent of $t$.
A similar result in the ultraparabolic Morrey spaces  was proved in \cite{PR_01}.
We point out that the papers \cite{MP_98} and \cite{PR_01} are concerned with the operators that are more general than $\cP$.
We also mention a recent paper \cite{AAMN_21} which studies the $L_2$-regularity theory, the trend to equilibrium, and enhanced dissipation for the KFP equation in divergence form with $a^{i j} = \delta_{i j}$.

\textit{Nondivergence form equations}

 For a thorough review of the classical theory for the generalized KFP equations, we refer the reader to \cite{AP_20}.
An overview of the literature on the Sobolev theory for KFP equations in nondivergence form can be found in  a recent paper \cite{DY_21}.
We also mention briefly the following papers:
\begin{itemize}
\item  \cite{BCM_96} on the interior $S_p$ estimate with leading coefficients of class $VMO$ with respect to $\rho$;

\item the articles \cite{BCLP_10}, \cite{CZ_19}, \cite{BCLP_13}, \cite{NZ_20} where the global $S_p$ estimate is proved under the following assumptions: either the leading coefficients are constant or independent of $x, v$, or they  are continuous with respect to $\rho$;

\item  \cite{DY_21}, where the present authors proved the global $S_p$ estimate and unique solvability results with the coefficients satisfying   Assumption \ref{assumption 2.2}.
\end{itemize}

\subsection{Organization of the paper}
In the first  section, we prove the main result in the $\bS_p$ space in the case when the coefficients $a^{i j}$ are independent of $x$ and $v$.
We then extend the a priori estimate to the weighted mixed-norm kinetic spaces in Section \ref{section 3} so that  the  reader    interested only in the constant coefficient case  needs to  read only the first three sections. We prove the main results for the equations with the variable coefficients $a^{i j} = a^{i j} (z)$ in Section \ref{section 4}.

\section{\texorpdfstring{$\bS_p$}{}-estimate for the model equation}

Denote
$$
    P_0 = \partial_t  - v \cdot D_x  - a^{i j} (t) D_{v_i v_j},
$$
where the coefficients $a^{i j}$ satisfy Assumption \ref{assumption 2.1}.

The goal of this section is to prove Theorem \ref{theorem 1.1} with  $L_p$  in place of  the weighted mixed-norm Lebesgue space, $\cP = P_0$ (see \eqref{1.1.0}), and without the lower-order terms (see Theorem \ref{theorem 11.3}).
We do this by using
the duality argument and the $S_p$-estimate taken from \cite{DY_21}, which we state below (see Theorem \ref{theorem 4.10}).

\begin{theorem}
                \label{theorem 11.3}
Let $p > 1$ be a number and
 $T \in (-\infty, \infty]$.
Then, the following assertions hold.

$(i)$
For any number $\lambda > 0$,  $u \in \bS_p (\bR^{1+2d}_T)$,
and $\vec f, g \in L_p (\bR^{1+2d}_T)$,
the equation
\begin{equation}
                \label{11.3.1}
    P_0 u + \lambda u = \div \vec f + g,
\end{equation}
 has a unique solution $u \in \bS_p (\bR^{1+2d}_T)$, and, in addition,
\begin{equation}
                \label{11.3.3}
 \begin{aligned}
&    \lambda^{1/2}  \|u\|_{ L_p (\bR^{1+2d}_T)  }	
+  \|D_v u\|_{ L_p (\bR^{1+2d}_T)  } 	  +   \|(-\Delta_x)^{1/6} u\|_{ L_p (\bR^{1+2d}_T)  }\\
  &  	\le N (d, \delta, p) (  \|\vec f\|_{ L_p (\bR^{1+2d}_T)  } + \lambda^{-1/2} \|g\|_{ L_p (\bR^{1+2d}_T)  }).
 \end{aligned}
 \end{equation}

$(ii)$
For any finite numbers $\lambda \ge 0$,
$S < T$, and
  $
  	\vec f, g \in L_p ((S, T) \times \bR^{2d}),
  $
the Cauchy problem
$$
   P_0 u + \lambda u =  \text{div}_v \vec f + g,
    \quad u (S, \cdot) \equiv 0
$$
has a unique solution
$
	u \in \bS_p ((S, T) \times \bR^{2d})
$
(see  Definition \ref{definition 2.1}), and,
furthermore,
\begin{align*}
     & \lambda^{1/2}  \|u\|_{ L_p ( (S, T) \times \bR^{2d}) }
   +\|D_v u\|_{ L_p ((S, T) \times \bR^{2d}) }
   +  \|(-\Delta_x)^{1/6}  u\|_{ L_p ((S, T) \times \bR^{2d}) }\\
&
        \leq N   ( \|\vec f\|_{  L_p ((S, T) \times \bR^{2d}) }+ \lambda^{-1/2} \|g\|_{  L_p ((S, T) \times \bR^{2d}) }),
\end{align*}
where $N = N (d, \delta, p, T-S)$.
\end{theorem}

\begin{corollary}
				\label{corollary 2.2}
For any $u \in \bS_{p} (\bR^{1+2d}_T)$,
one has $(-\Delta_x)^{1/6} u \in L_{p} (\bR^{1+2d}_T)$, and, in addition,
\begin{align*}
	\|(-\Delta_x)^{1/6} u\|_{ L_{p} (\bR^{1+2d}_T) }
	\le  N \|u\|_{\bS_{p} (\bR^{1+2d}_T)},
\end{align*}
where $N  = N (d, p) > 0$.
\end{corollary}
\begin{proof}
To prove the result, we set $a^{i j} = \delta_{i j}$, find $\vec f$ and $g$ in $L_{p} (\bR^{1+2d}_T)$ such that
$$
	\partial_t u - v \cdot D_x u - \Delta_v u + u = \div \vec f + g
$$
in $\mathbb{H}^{-1}_{p}  (\bR^{1+2d}_T)$,
and apply Theorem \ref{theorem 11.3}.
\end{proof}

Here is the main result of \cite{DY_21} in the case when $\cP = P_0$, which will also be used in the next section.

\begin{theorem}
			\label{theorem 4.10}
Let $p > 1$ be number. Then, the following assertions hold.

$(i)$
For any number $\lambda \ge 0$, $T \in (-\infty, \infty]$, and
$u \in S_{p} (\bR^{1+2d}_T)$,
one has
\begin{equation*}
    \begin{aligned}
&	\lambda \|u \|
+\lambda^{1/2}\|D_v u\|
+\|D^2_v u\|\\
&+ \|(- \Delta_x)^{1/3} u \|
+  \|D_v (- \Delta_x)^{1/6} u \|
\leq N (d, p, \delta)  \| P_0 u   + \lambda  u \|,
    \end{aligned}
\end{equation*}
where $\|\cdot\| = \|\cdot\|_{  L_{p} (\bR^{1+2d}_T) }$.

$(ii)$
For any $\lambda > 0$,
$T \in (-\infty, \infty]$,
and
$f \in L_{p} (\bR^{1+2d}_T)$,
the equation
\begin{equation*}
    P_0 u + \lambda u = f
\end{equation*}
has a unique solution
$u \in S_{p} (\bR^{1+2d}_T)$.

$(iii)$ For any finite numbers $S < T$   and
  $  f \in L_p ((S, T) \times \bR^{2d})$,
the Cauchy problem
$$
    P_0 u  =  f , \quad u (S, \cdot) \equiv 0,
$$
has a unique solution
$u \in
S_p((S, T) \times \bR^{2d})$.
In addition,
\begin{align*}
&    \|u\|
+  \|D_v u\|
+  \|D^2_v u\|
       +\|(-\Delta_x)^{1/3}  u\|
+ \|D_v (- \Delta_x)^{1/6} u \|
\leq N\|f\|,
\end{align*}
where $$
    \|\cdot\| = \|\cdot\|_{ L_{p} ((S, T) \times \bR^{2d}) }, \quad N = N (d, \delta, p, T-S).
    $$
\end{theorem}

\begin{remark}
The above theorem follows from Theorem  2.6 of \cite{DY_21} and the scaling property of the operator $P_0$ (see Lemma \ref{lemma 4.1}).
\end{remark}

The following lemma implies the uniqueness part of Theorem \ref{theorem 11.3} $(ii)$.
\begin{lemma}
            \label{lemma 12.2}
Let $p > 1$, $\lambda > 0$ be numbers, $T \in (-\infty, \infty]$, and $u \in \bS_p (\bR^{1+2d}_T)$ satisfies $P_0 u + \lambda u = 0$.
Then, $u \equiv 0$.
\end{lemma}

\begin{proof}
Let $\eta  = \eta (x, v) \in C^{\infty}_0 (B_1 \times B_1)$ be a function with the unit integral.
For $h \in L_{1, \text{loc}} (\bR^{2d})$,
we denote
\begin{equation*}
h_{(\varepsilon)} (x, v) =  \varepsilon^{-(3/2)d} \int h (x', v') \eta ((x-x')/\varepsilon^{1/2}, (v -v')/\varepsilon) \, dx'dv'.
\end{equation*}
Then, $u$ satisfies the equation
\begin{equation}
                \label{12.4.1}
    P_0 u_{(\varepsilon)} + \lambda u_{(\varepsilon)} =
     g_{\varepsilon},
\end{equation}
where
\begin{equation*}
    g_{\varepsilon} (z) =
    \varepsilon^{1/2} \int u (t, x - \varepsilon^{1/2} x', v - \varepsilon v') v' \cdot D_x \eta (x', v') \,dx'dv'.
\end{equation*}
Note that by the Minkowski inequality,
\begin{equation}
                \label{12.4.3}
    \|g_{\varepsilon}\|_{ L_p (\bR^{1+2d}_T)  }
    \le N \varepsilon^{1/2} \|u\|_{ L_p (\bR^{1+2d}_T)  }.
\end{equation}
Then, it follows from \eqref{12.4.1} that
$
(\partial_t - v \cdot D_x) u_{(\varepsilon)} \in L_{p}(\bR^{1+2d}_T),
$
and, therefore,
$
    u_{(\varepsilon)} \in
    S_{p}(\bR^{1+2d}_T).
$
Hence, by Theorem \ref{theorem 4.10}
and \eqref{12.4.3},
$$
  \lambda \| u_{(\varepsilon)}\|_{ L_{p}(\bR^{1+2d}_T) }
   \le  N \| g_{\varepsilon}\|_{ L_{p}(\bR^{1+2d}_T) }
    \le N \varepsilon^{1/2}
    \| u\|_{ L_{p}(\bR^{1+2d}_T) }.
 $$
Taking the limit as $\varepsilon \to 0$ in the above inequality, we prove the assertion.
\end{proof}

The following result is needed for the duality argument in the proof of Theorem \ref{theorem 11.3}. For the proof, see Lemma 5.12 of \cite{DY_21}.

\begin{lemma}
            \label{lemma 4.15}
For any numbers $\lambda \ge 0$ and $p >1$,
the set $(P_0 + \lambda) C^{\infty}_0 (\bR^{1+2d})$
is dense in $L_p (\bR^{1+2d})$.
\end{lemma}

Here is the a priori estimate \eqref{11.3.3} in the case when $g \equiv 0$ and $\vec f$ is smooth and compactly supported.
\begin{lemma}
            \label{lemma 11.4}
Let $\lambda > 0, p > 1$ be numbers, and $\vec f \in C^{\infty}_0 (\bR^{1+2d})$.
Let $u$ be the unique solution in $S_p (\bR^{1+2d})$ to Eq. \eqref{11.3.1} with $g \equiv 0$. Then, one has
 \begin{align*}
   & \lambda^{1/2}  \|u\|_{ L_p (\bR^{1+2d})  }	
    +\|(-\Delta_x)^{1/6} u\|_{ L_p (\bR^{1+2d})  }	\\
  &  + \|D_v u\|_{ L_p (\bR^{1+2d})  }
	\le N (d, \delta, p) \|\vec f\|_{ L_p (\bR^{1+2d})  }.
 \end{align*}	
\end{lemma}

\begin{proof}
Proof by duality argument.
We denote $q  = p/(p-1)$ and fix some $U \in C^{\infty}_0 (\bR^{1+2d})$.

 \textit{Estimate of $(-\Delta_x)^{1/6} u$.}
Note that for any multi-index $\alpha$, one has $P_0 D_x^{\alpha} u  = \text{div}_v  D_x^{\alpha} \vec f$,
and hence, by Theorem \ref{theorem 4.10}, $D_x^{\alpha} u \in L_p (\bR^{1+2d})$.
In addition, note that by \eqref{eq1.15} for any $U \in C^{\infty}_0 (\bR^{1+2d})$, $(-\Delta_x)^{1/6} U \in C^{\infty}_{\text{loc}} (\bR^{1+2d}) \cap L_{1} (\bR^{1+2d})$.
Then, by this and  integration by parts, we have
$$
    I = \int \big((-\Delta_x)^{1/6} u \big)
    (-\partial_t U + v \cdot D_x U - a^{i j} (t) D_{v_i v_j} U + \lambda U) \, dz
$$
 $$
    = \int    \big((\partial_t  - v \cdot D_x  - a^{i j} (t) D_{v_i v_j}  + \lambda)u\big) ((-\Delta_x)^{1/6} U) \, dz
$$
 $$
    =
    \int  (\text{div}_v \vec f) ((-\Delta_x)^{1/6} U)  \, dz
    = -  \int  \vec f \cdot D_v (-\Delta_x)^{1/6} U \, dz.
$$
By H\"older's inequality, Theorem \ref{theorem 4.10}, and the change of variables $t \to -t$, $x \to -x$,
\begin{equation}
                \label{11.4.1}
    |I|
    \le
   N \|\vec f\|_{ L_p (\bR^{1+2d}) }
    \|-\partial_t U + v \cdot D_x U  - a^{i j}  D_{v_i v_j} U + \lambda U\|_{ L_q (\bR^{1+2d}) }.
 \end{equation}
Furthermore, by Lemma \ref{lemma 4.15}  and the aforementioned change of variables,
$(-\partial_t  + v \cdot D_x - a^{i j} D_{v_i v_j} + \lambda) C^{\infty}_0 (\bR^{1+2d})$ is dense in $L_q (\bR^{1+2d})$.
This combined with \eqref{11.4.1} implies the desired estimate for $(-\Delta_x)^{1/6} u$.

\textit{Estimate of $D_v u$}.
Integrating by parts gives
\begin{align*}
  &\int   ( D_v u )   (-\partial_t U + v \cdot D_x U - a^{i j} (t) D_{v_i v_j} U + \lambda U) \, dz\\
 &
	  = \int  (D_v^2 U)  \vec f \,  dz - \int u D_x U \, dz
	  =: J_1 + J_2.
  \end{align*}
  As before, it suffices
  to show that
  $|J_1| + |J_2|$ is dominated
  by the right-hand side of
  \eqref{11.4.1}.
  By H\"older's inequality,
   Theorem \ref{theorem 4.10}, and the same change of variables,
   we get
 \begin{align*}
  |J_1|
    &\le
     \|\vec f\|_{ L_p (\bR^{1+2d}) }
    \|D_{v}^2 U\|_{ L_q (\bR^{1+2d}) }\\
  &
    \le N  \|\vec f\|_{ L_p (\bR^{1+2d}) }
    \|-\partial_t U + v \cdot D_x U - a^{i j} (t) D_{v_i v_j} U + \lambda U\|_{ L_q (\bR^{1+2d}) }.
   \end{align*}
  Next, note that
  $$
    J_2 = - \int ((-\Delta_x)^{1/6} u) \, \cR_x (-\Delta_x)^{1/3} U \, dz,
  $$
  where $\cR_x$ is the Riesz transform in the $x$ variable.
  Then, by the $L_p$ estimate of $(-\Delta_x)^{1/6} u$, the $L_{q}$ boundedness of the Riesz transform, and Theorem \ref{theorem 4.10},
  we obtain
  \begin{align*}
    |J_2| &\le N  \|(-\Delta_x)^{1/6} u\|_{ L_{p} (\bR^{1+2d}) } \|(-\Delta_x)^{1/3} U\|_{ L_q (\bR^{1+2d}) }\\
  &
    \le N \|\vec f\|_{ L_p (\bR^{1+2d}) }
    \|-\partial_t U + v \cdot D_x U - a^{i j}  D_{v_i v_j} U + \lambda U\|_{ L_q (\bR^{1+2d})}.
   \end{align*}
   The estimate is proved.

 \textit{Estimate of $u$}.
As  above, we consider
$$
   \cI: =  \int u (-\partial_t U + v \cdot D_x U - a^{i j} D_{v_i v_j} U + \lambda U) \, dz
   = - \int \vec f \cdot D_v U \, dz.
 $$
 Then, by H\"older's inequality and
  Theorem \ref{theorem 4.10},
\begin{align*}
        | \cI| & \le \|\vec f \|_{ L_p (\bR^{1+2d})} \|D_v U\|_{ L_q (\bR^{1+2d})  }\\
&
    \le N \lambda^{-1/2}
    \|\vec f \|_{ L_p (\bR^{1+2d})  }
    \|-\partial_t U + v \cdot D_x U -  a^{i j} D_{v_i v_j}  U + \lambda U\|_{ L_q (\bR^{1+2d})  }.
 \end{align*}
 This implies the desired estimate.
\end{proof}

\begin{proof}[Proof of Theorem \ref{theorem 11.3}.]
By Remark \ref{remark 2.2}, we only need to prove the assertion  $(i)$.

$(i)$ The uniqueness follows from  Lemma \ref{lemma 12.2}.
To prove the existence, let $u_1 \in S_p (\bR^{1+2d}_T)$  be the unique solution to the equation (see Theorem \ref{theorem 4.10})
\begin{equation}
                \label{eq11.3.3}
        P_0 u_1 + \lambda u_1 = g.
\end{equation}
By the same  theorem  and the interpolation inequality (see Lemma \ref{lemma A.5}),
\begin{align*}
&    \lambda \|u_1\|_{ L_p (\bR^{1+2d}_T) } + \lambda^{1/2}\|D_v u_1\|_{ L_p (\bR^{1+2d}_T) } + \lambda^{1/2}\|(-\Delta_x)^{1/6} u_1\|_{ L_p (\bR^{1+2d}_T) }\\
&\le N \|g\|_{ L_p (\bR^{1+2d}_T) }.
\end{align*}
Subtracting Eq. \eqref{eq11.3.3} from Eq. \eqref{11.3.1}, we may assume that $g \equiv 0$.
We will consider the cases $T = \infty$ and $T < \infty$ separately.

\textit{Case $T = \infty$.}
We take a sequence of functions
$\vec f_n \in C^{\infty}_0 (\bR^{1+2d})$
such that $f_n \to f$  in $L_p (\bR^{1+2d})$.
By Theorem \ref{theorem 4.10}, there exists
a unique solution $u_n \in S_p (\bR^{1+2d})$ to the equation
\begin{equation}
                \label{11.3.2}
    (P_0 + \lambda) u_n  = \text{div}_v \vec f_n.
\end{equation}
By  Lemma \ref{lemma 11.4}, we have
    \begin{align}
                            \label{11.3.4}
   & \lambda^{1/2}  \|u_n\|_{ L_p (\bR^{1+2d})  }
   +  \|D_v u_n\|_{ L_p (\bR^{1+2d})  } 	
+ \|(-\Delta_x)^{1/6} u_n\|_{ L_p (\bR^{1+2d})  }	\\
  &
	\le N (d, \delta, p)
	\|\vec f_n\|_{ L_p (\bR^{1+2d})  },\notag\\
&
\|\partial_t u_n - v \cdot D_x u_n\|_{ \bH^{-1}_p (\bR^{1+2d})  } \le N (d, \delta, p,\lambda)
	\|\vec f_n\|_{ L_p (\bR^{1+2d})  }\notag.
 \end{align}
Furthermore, by the same lemma,
$$
u_n, n \ge 1, \quad (-\Delta_x)^{1/6} u_n,\, n \ge 1
$$
are Cauchy sequences in $\bS_p (\bR^{1+2d})$ and $L_p (\bR^{1+2d})$ , respectively.
Hence, there exists a function $u\in \bS_p(\bR^{1+2d})$ such that $u_n$,  $(-\Delta_x)^{1/6} u_n$ converge to $u$ and $(-\Delta_x)^{1/6} u$, respectively.
Passing to the limit in Eq. \eqref{11.3.2} and \eqref{11.3.4}, we prove the existence and the inequality \eqref{11.3.3}.

\textit{Case $T < \infty$.} \label{T finite}
Let $\widetilde u \in \bS_p (\bR^{1+2d})$ be the unique solution to the equation
$$
    P_0 \widetilde u + \lambda \widetilde u = \text{div}_v \vec f 1_{t < T}.
$$
We conclude that $u: = \widetilde u$ is a  solution of class $\bS_p (\bR^{1+2d}_T)$ to Eq. \eqref{11.3.1},
and the estimate \eqref{11.3.3} holds.
The theorem is proved.
\end{proof}

\mysection{Mixed-norm estimate for the model equation}
											\label{section 3}

In this section, we consider the case when the coefficients $a^{i j}$ are independent of $x, v$,  and the lower-order terms are absent.
The goal is  to prove the a  priori  estimates in the weighted mixed-norm  spaces by  establishing a mean oscillation estimate of $(-\Delta_x)^{1/6} u$, $\lambda^{1/2} u$, and $D_v u$ for $u \in \bS_p (\bR^{1+2d}_T)$  solving  Eq. \eqref{11.3.1}. To this end,  we split $u$ into a $P_0+\lambda$-caloric part and the remainder. To bound the former, we use the method of Section 5 of \cite{DY_21}. The remainder is handled by using a localized version of the $\bS_p$ estimate in Theorem \ref{theorem 11.3} (see Lemma \ref{lemma 12.1}).

\begin{theorem}
            \label{theorem 12.1}
Invoke the assumptions of  Theorem \ref{theorem 1.1}  and assume, additionally, $b \equiv 0 \equiv \overline{b}$, $c \equiv 0$.
Let $u \in \bS_{p, r_1, \ldots, r_d, q} (\bR^{1+2d}_T, w)$, $\vec f , g\in L_{p, r_1, \ldots, r_d, q} (\bR^{1+2d}_T, w)$ be functions such that
$$
    P_0 u  + \lambda u = \div \vec f +g.
$$
Then, for any $\lambda > 0$, the estimate \eqref{2.1.2} is valid.
Furthermore, in the case when $g \equiv 0$ and $\lambda \equiv 0$, \eqref{2.1.2} also holds.
In addition, the same inequalities hold with $\bS_{p; r_1,\ldots, r_d} (\bR^{1+2d}_T, |x|^\alpha \prod_{i = 1}^d w_i (v_i))$, $L_{p; r_1,\ldots, r_d} (\bR^{1+2d}_T, |x|^\alpha \prod_{i = 1}^d w_i (v_i))$  in place of
$\bS_{p, r_1,\ldots, r_d, q} (\bR^{1+2d}_T, w)$,
$L_{p, r_1,\ldots, r_d, q} (\bR^{1+2d}_T, w)$, respectively,
and with \\
$N = N (d, \delta,  p, r_1, \ldots, r_d,  \alpha, K)$.
\end{theorem}

The next result is derived from the above theorem in the way as  Corollary \ref{corollary 2.2} from Theorem \ref{theorem 11.3}.
\begin{corollary}[cf. Corollary \ref{corollary 2.2}]
				\label{corollary 3.2}
For any $u \in \bS_{p, r_1, \ldots, r_d, q} (\bR^{1+2d}_T, w)$,
one has\\
$(-\Delta_x)^{1/6} u \in L_{p, r_1, \ldots, r_d, q} (\bR^{1+2d}_T, w)$, and, in addition,
\begin{align*}
	\|(-\Delta_x)^{1/6} u\|_{ L_{p, r_1, \ldots, r_d, q} (\bR^{1+2d}_T, w) }
	\le  N \|u\|_{\bS_{p, r_1, \ldots, r_d, q} (\bR^{1+2d}_T, w)},
\end{align*}
where $N  = N (d, p, r_1, \ldots, r_d, q, K) > 0$.
  A similar assertion holds for \\
$u \in \bS_{p; r_1,\ldots, r_d} (\bR^{1+2d}_T, |x|^\alpha \prod_{i = 1}^d w_i (v_i))$.
\end{corollary}

In the next lemma, we establish the estimate  of the aforementioned ``remainder term".
\begin{lemma}
            \label{lemma 12.1}
Let $\lambda \ge 0$ be a number, $\vec f\in L_p(\bR^{1+2d}_0)$ be a function vanishing outside  $(-1,0)\times \bR^d\times B_1$.
Let $u\in \bS_p((-1,0)\times \bR^{2d})$ be the unique solution to the equation
(see Theorem \ref{theorem 11.3} $(ii)$)
\begin{equation}
            \label{12.1}
	P_0 u +\lambda u =  \div \vec f +g,
	\quad
	u (-1, \cdot) = 0.
\end{equation}
Then, for any $R \ge 1$, one has
\begin{align}
            \label{12.2}
&\|(1+\lambda^{1/2})|u|+|D_v u|\|_{L_p((-1, 0) \times B_{R^3} \times B_R)}\notag\\
&\le N (d, \delta, p) \sum_{k=0}^\infty 2^{-k(k-1)/4} R^{-k}
\| |\vec f| +  \lambda^{-1/2} |g| \|_{L_p(Q_{1,2^{k+1} R})},
\\
&
    \label{12.3}
(|(-\Delta_x)^{1/6} u|^p)^{1/p}_{Q_{1, R}}
	\le N (d, \delta, p)  \sum_{k=0}^\infty 2^{-k}  \big( (|\vec f|^p)^{1/p}_{Q_{1,2^{k +1 } R}}
	+\lambda^{-1/2} (|g|^p)^{1/p}_{Q_{1,2^{k +1} R}}\big).
\end{align}
\end{lemma}

\begin{proof}
We follow the proof of Lemma 5.2 of \cite{DY_21} very closely. By considering the equation satisfied by $U:=u e^{-t}$, without loss of generality, we may assume that $\lambda\ge 1$.

\textit{Estimate of $u, D_v u$.}
Denote
$$
   \vec  f_0 := \vec f 1_{x \in B_{(2R)^3}}, \, \,   \vec f_k := \vec f 1_{x \in B_{ (2^{k+1} R)^3 } \setminus B_{ (2^{k} R)^3 } },\, k  \in \{1, 2, \ldots\}, \, \, \text{so that} \, \, \vec f = \sum_{k = 0}^{\infty} \vec f_k,
$$
and we define $g_k, k \ge 0$ in a similar way. By Theorem \ref{theorem 11.3} $(ii)$,
there exists a unique solution $u_k \in \bS_p ((-1, 0)\times \bR^{2d})$ to Eq. \eqref{12.1} with $\vec f_k$ and $g_k$ in place of $\vec f$ and $g$, respectively, and, in addition,  one has
\begin{align}
            \label{12.4}
&\|\lambda^{1/2}|u_k|+|D_v u_k|\|_{L_p ((-1,0)\times \bR^{2d})}
\le N  \| |\vec f_{k}| +\lambda^{-1/2}|g_k|\|_{L_p ((-1,0)\times \bR^{2d})}.
\end{align}
In addition, by Theorem \ref{theorem 11.3} $(ii)$,
$$
   u = \lim_{n \to \infty} \sum_{k = 0}^n  u_k \quad \text{in} \, \,  L_{ p } ((-1, 0)\times \bR^{2d}),
$$
and a similar identity holds for  $D_v u$.

Next, let $\zeta_j= \zeta_j (x,v) \in C_0^\infty(B_{(2^{j+1} R)^3}\times B_{2^{j+1} R})$, $j=0,1,2,\ldots,$ be a sequence of functions such that $\zeta_j=1$ on $B_{(2^{j+1/2} R)^3}\times B_{2^{j+1/2} R}$ and
\begin{align*}
&|\zeta_j|\le 1,\quad |D_v \zeta_j|\le N 2^{-j}R^{-1},\\
&|D^2_v \zeta_j|\le N 2^{-2j}R^{-2},
\quad |D_x \zeta_j|\le N  2^{-3j}R^{-3}.
\end{align*}
For $k\ge 1$ and $j=0,1,\ldots,k-1$, we set $u_{k,j}=u_k\zeta_j$, which satisfies
\begin{equation}
			\label{eq12.1.1}
P_0 u_{k,j}  + \lambda u_{k,j}= u_k P_0 \zeta_j + \div ( \vec f_k \zeta_j)  -\vec f_k \cdot D_v \zeta_j + g_k \zeta_j -2(aD_v \zeta_j)\cdot D_v u_k.
\end{equation}
Observe that for such $j$,  $\vec f_k \zeta_j\equiv 0$, $\vec f_k \cdot D_v \zeta_j \equiv 0$, and $g_k \zeta_j = 0$.
Then, by Theorem \ref{theorem 11.3} $(ii)$ and the fact that $\lambda \ge 1$,
\begin{align*}
&\|\lambda^{1/2}|u_{k}|+|D_v u_{k}|\|_{L_p((-1,0)\times B_{(2^j R)^3}\times B_{2^j  R})}\\
&\le N  \lambda^{-1/2} 2^{-j} R^{-1}\||u_{k}|+|D_v u_{k}|\|_{L_p((-1,0)\times B_{(2^{j+1} R)^3}\times B_{2^{j+1} R})}\\
&\le N 2^{-j} R^{-1} \|\lambda^{1/2} |u_{k}|+|D_v u_{k}|\|_{L_p((-1,0)\times B_{(2^{j+1} R)^3}\times B_{2^{j+1} R})}.
\end{align*}
By using induction, the above inequality, and \eqref{12.4}, we obtain
\begin{align*}
&\| \lambda^{1/2} |u_{k}|+|D_v u_{k}|\|_{L_p((-1,0)\times B_{ R^3}\times B_{R})}\\
& \le N^k 2^{-k(k-1)/2} R^{-k} \| |\vec f_{k}| + \lambda^{-1/2}|g_k|\|_{L_p((-1, 0) \times \bR^{2d})}\\
& \le N 2^{-k(k-1)/4} R^{-k}\||\vec f|+ \lambda^{-1/2}|g|\|_{L_p(Q_{1, 2^{k+1}R})}.
\end{align*}
This combined with \eqref{12.4} with $k = 0$
gives \eqref{12.2}.

\textit{Estimate of $(-\Delta_x)^{1/6} u$}.
Recall that $u \zeta_0$ satisfies \eqref{eq12.1.1} with $k = 0.$
Then, by Theorem \ref{theorem 11.3} $(ii)$ and \eqref{12.2} with $2 R$ in place of $R$, one has
\begin{equation}
            \label{eq12.1.2}
\begin{aligned}
&\|(-\Delta_x)^{1/6} (u\zeta_0)\|_{L_p((-1,0)\times \bR^{2d})}\\
&	\le
N\sum_{k=0}^\infty 2^{-k(k-1)/4} (2R)^{-k}
 \| |\vec f|+ \lambda^{-1/2}|g|\|_{L_p(Q_{1,2^{k +1} (2R)  })}.
\end{aligned}
\end{equation}
Now we only need to bound the commutator term.
Let $u_{\varepsilon}$ be a mollification of $u$ in the $x$ variable.
Then, $(-\Delta_x)^{1/6} u_{\varepsilon}$ is given by \eqref{eq1.20} with $s = 1/6$.
Due to  the fact that $\zeta_0 = 1$ in $B_{(2^{1/2} R)^3} \times B_{2^{1/2} R }$, for any $z \in Q_{1, R}$,
\begin{align*}
&|   \zeta_0  (-\Delta_x)^{1/6} u_{\varepsilon} - (-\Delta_x)^{1/6} (u_{\varepsilon}  \zeta_0 )| (z)\\
&\le N (d) \int_{ |y| > (2^{3/2} - 1) R^3 } |u_{\varepsilon}| (t, x+y, v) |y|^{-d-1/3} \, dy.
\end{align*}
Then, by Lemma \ref{lemma 4.7}, we get
\begin{align*}
    &\|\zeta_0 (-\Delta_x)^{1/6} u_{\varepsilon} - (-\Delta_x)^{1/6} (u_{\varepsilon} \zeta_0)\|_{ L_p (Q_{1, R}) }\\
&    \le N (d) R^{-1} \sum_{k = 0}^{\infty} 2^{-k - 3d k/p } \|u_{\varepsilon}\|_{ L_{p} (Q_{1, 2^k R}) }.
\end{align*}
Passing to the limit as $\varepsilon \to 0$, we may replace $u_{\varepsilon}$ with $u$ in the above inequality. Furthermore, by \eqref{12.2}, the right-hand side is less than
 $$
     N (d, \delta, p) R^{-1} \sum_{j = 0}^{\infty}
    2^{-k - 3d k/ p }
    \sum_{k = 0}^{\infty}  2^{-j (j-1)/4} (2^k R)^{-j} \|  |\vec f|+ \lambda^{-1/2}|g|\|_{ L_{  p } (Q_{1, 2^{k+j+1} R}) }.
 $$
Switching the order of summation and changing the index $k \to k+j$, we may replace the double sum with
 $$
	\sum_{k = 0}^{\infty} 2^{- k -3dk/p} \| |\vec f|+ \lambda^{-1/2}|g|\|_{ L_{ p } (Q_{1, 2^{k+1 } R}) }.
 $$
This combined with
\eqref{eq12.1.2} gives the desired estimate \eqref{12.3}.
\end{proof}

Here is the mean oscillation estimate of a $P_0+\lambda$-caloric part.
\begin{proposition}
            \label{proposition 12.3}
Let $p > 1$, $\lambda \ge  0$, $r > 0$, $\nu \geq 2$ be numbers,
 $z_0 \in  \overline{\bR^{1+2d}_T}$,
and
$u \in \bS_{p} ((t_0-(4 \nu r)^2, t_0) \times \bR^{2d})$ be a function such that
$P_0 u +\lambda u = 0$
 in $(t_0- (\nu r)^2, t_0) \times \bR^d \times B_{\nu r} (v_0)$.
Then, one has
  \begin{align*}
    J_1 &:=
    \bigg(|(-\Delta_x)^{1/6} u -  ((-\Delta_x)^{1/6} u)_{ Q_r (z_0) }|^p\bigg)^{1/p}_{Q_r (z_0)}\\
&\leq N  \nu^{-1}
    (|(-\Delta_x)^{1/6} u|^p)^{1/p}_{ Q_{\nu r } (z_0)},\\
       	J_2 &:= \lambda^{1/2}\bigg(|u -  (u)_{ Q_r (z_0) }|^p\bigg)^{1/p}_{ Q_r (z_0) }
 	+ \bigg(|D_v u -  (D_v u)_{ Q_r (z_0) }|^p\bigg)^{1/p}_{ Q_r (z_0) } \\
&
  \leq  N   \nu^{-1}
    \lambda^{1/2} (|u|^p)^{1/p}_{ Q_{\nu r} (z_0)  }
+ N   \nu^{-1}
    (|D_v u|^p)^{1/p}_{ Q_{\nu r} (z_0)  }\\
    &\quad + N \nu^{-1}
    \sum_{k = 0}^{\infty}
	 2^{-2k}
	 (|(-\Delta_x)^{1/6} u|^p)^{1/p}_{ Q_{\nu r, 2^k \nu r} (z_0)},
    \end{align*}
where $N = N (d, \delta, p)$.
\end{proposition}

\subsection{Proof of Proposition \ref{proposition 12.3}}

The next two lemmas are taken from   \cite{DY_21}. The first one, Lemma \ref{lemma 5.1}, is proved by localizing Theorem \ref{theorem 4.10} $(i)$.
The second lemma follows from the global $L_p$  estimate of $(-\Delta_x)^{1/3} u$
in Theorem \ref{theorem 4.10} $(i)$
and the local estimate
of $D_v u$ in Lemma \ref{lemma 5.1}. 
\begin{lemma}[Interior $S_p$ estimate, see Lemma 6.4 of \cite{DY_21}]
            \label{lemma 5.1}
Let  $p > 1$, $\lambda \geq 0$, and $r_1, r_2, R_1, R_2 > 0$
 be  numbers such that
$r_1 < r_2$ and $R_1 < R_2$.
Let
$u \in S_{p,\text{loc}} (\bR^{1+2d}_0)$ and
denote $f = P_0 u + \lambda u$.
Then, there exists a constant $N = N (d, \delta, p) > 0$ such that
\begin{align*}
	& \,
	  \lambda \|u\|_{ L_p ( Q_{r_1,  R_1}) } +
	  (r_2 - r_1)^{-1} \|D_v u\|_{ L_p (Q_{r_1,  R_1}) }\\
	   &+ \|D^2_v u\|_{ L_p ( Q_{r_1,  R_1}) }
	   + \|\partial_t u - v \cdot D_x u\|_{ L_p ( Q_{r_1,  R_1}) }\\
	&\leq N   \| f \|_{ L_p ( Q_{r_2,  R_2}) }+  N   ((r_2 - r_1)^{-2} + r_2 (R_2 - R_1)^{-3})  \|u\|_{ L_p ( Q_{r_2,   R_2}) }.
\end{align*}	
\end{lemma}

\begin{lemma}[Caccioppoli type inequality, see Lemma 6.5 of \cite{DY_21}]
            \label{lemma 5.2}
Let $\lambda \ge 0$,
$0 < r < R \leq 1$, and $p > 1$ be numbers,
and
$u\in S_{p,\text{loc}}(\bR^{1+2d}_0)$ be a function
such that $P_0 u +\lambda u = 0$ in $Q_1$.
Then, there exists a constant
$N = N (d,\delta, p,r, R)$ such that
\begin{equation}
                \label{eq5.2.1}
    \|D_x u\|_{L_p (Q_r)}
    \leq N  \|u\|_{L_p (Q_R)}.
\end{equation}
\end{lemma}

\begin{remark}
In the interior estimates in the aforementioned Lemma 6.4 of \cite{DY_21}, there are no terms involving $\lambda u$  and $\partial_t u - v \cdot D_x u$. By following the proof of that lemma and using the global $S_p$ estimate (see Theorem \ref{theorem 4.10}), one can, indeed, add these terms to the left-hand sides of the a priori estimates.

Furthermore, the Caccioppoli inequality in Lemma 6.5 of \cite{DY_21} is stated only in the case when $\lambda = 0$. Nevertheless, the same argument yields \eqref{eq5.2.1} in the case when $\lambda > 0$.
\end{remark}

The next lemma is a key ingredient of the proof of Proposition \ref{proposition 12.3}.
\begin{lemma}[cf. Lemma 6.6 of \cite{DY_21}]
            \label{lemma 5.3}
Let $p\in (1,\infty)$ and $u\in \bS_p ((-4, 0) \times \bR^d \times B_2)$
be a function such that
$P_0 u +\lambda u= 0$
in $(-1, 0) \times \bR^d \times B_1$.
Then, the following assertions hold.

$(i)$ The functions $u, (-\Delta_x)^{1/6} u \in S_{p, \text{loc}} ((-1, 0) \times \bR^d \times B_1)$. Furthermore,
\begin{equation}
                \label{eq5.3.12}
 (P_0 +  \lambda) u = 0, \quad   (P_0 + \lambda) (-\Delta_x)^{1/6} u = 0 \quad \text{a.e. in} \, \, (-1, 0)\times \bR^d \times B_1.
\end{equation}

$(ii)$ For any $r \in (0, 1)$, we have
\begin{equation}
                \label{eq5.3.0}
\begin{aligned}
  \|D_x u \|_{ L_p (Q_r)} &\leq
    N  \sum_{k = 0}^{\infty}
	  2^{- 2 k }	  (|(-\Delta_x)^{1/6} u|^p)_{ Q_{1, 2^k}}^{ 1/p   },
\end{aligned}
\end{equation}
where $N = N (d, \delta, p, r)$.
\end{lemma}

\begin{proof}
Multiplying $u \in \bS_p ((-4, 0)\times \bR^{d}\times B_2)$ by a cutoff function $\phi = \phi (t, v)$ and using Corollary \ref{corollary 2.2},
we conclude that $(-\Delta_x)^{1/6} u \in L_p ((-1, 0) \times \bR^{d} \times B_1)$, so that the series on the right-hand side of \eqref{eq5.3.0} converges.

$(i)$ Let $u_{\varepsilon}$ be the mollification of $u$ in the $x$ variable. First, we will show that $u_{\varepsilon}$ is sufficiently regular.
We fix some $r_0 \in (0, 1)$. We claim that for any $k = \{0, 1, 2, \ldots\}$,
\begin{equation}
            \label{eq5.3.10}
    D^k_x \xi  \in L_p ((-r_0^2, 0) \times \bR^d \times B_{r_0}) \quad \text{for} \, \, \xi = u_{\varepsilon}, \partial_t u_{\varepsilon}, D^2_v u_{\varepsilon}.
\end{equation}
To show this, we note that
$$
    P_0 u_{\varepsilon} + \lambda u_{\varepsilon} = 0 \quad (-1, 0) \times \bR^d \times B_1.
$$
We fix $x \in \bR^d$ and write
$$
    \partial_t u_{\varepsilon} - a^{i j} D_{v_i v_j} u_{\varepsilon} + \lambda u_{\varepsilon} = v\cdot D_x u_{\varepsilon} =: \mathsf{f} \quad \text{in} \, \, (-1, 0) \times B_{1}.
$$
For $f \in L_{1, \text{loc}} (\bR^d)$, let $f^{\varkappa}$ be a mollification of $f$ in the   $v$ variable.
Then, $u_{\varepsilon}^{\varkappa}$ satisfies
$$
    \partial_t u_{\varepsilon}^{\varkappa} - a^{i j} D_{v_i v_j} u_{\varepsilon}^{\varkappa} + \lambda u_{\varepsilon}^{\varkappa} =  \mathsf{f}^{\varkappa} \quad \text{in} \, \, (-r_0^2, 0) \times B_{r_0}, \quad  \varkappa \in (0, 1 - r_0).
$$
Note that $(t, v) \to \mathsf{f} \in L_p ((-1, 0) \times B_{1})$, and then,
$\partial_t u_{\varepsilon}^{\varkappa}, D^2_v u_{\varepsilon}^{\varkappa} \in L_p ((-r_0^2, 0) \times B_{r_0})$.
By the interior estimate for nondegenerate parabolic equations (cf. Lemma 2.4.4 in \cite{Kr_08}), for any $r_1 \in (0, r_0)$,
\begin{align*}
  &  \|\lambda |u_{\varepsilon}^{\varkappa} (\cdot, x, \cdot)| +     |\partial_t u_{\varepsilon}^{\varkappa}(\cdot, x, \cdot)| + |D^2_v u_{\varepsilon}^{\varkappa}(\cdot, x, \cdot)|\|_{L_p ((-r_1^2, 0) \times B_{r_1})}\\
  &\le N (d, \delta, p, r_0, r_1) \||\mathsf{f}^{\varkappa} (\cdot, x, \cdot)| + |u_{\varepsilon}^{\varkappa} (\cdot, x, \cdot)|\|_{L_p ((-r_0^2, 0) \times B_{r_0})}.
\end{align*}
Raising the above inequality to the power $p$ and integrating over $x \in \bR^d$, we get
\begin{align*}
  &  \|\lambda |u_{\varepsilon}^{\varkappa}| +     |\partial_t u_{\varepsilon}^{\varkappa}| + |D^2_v u_{\varepsilon}^{\varkappa}|\|_{L_p ((-r_1^2, 0) \times \bR^d \times B_{r_1})}\\
  &\le N   \||\mathsf{f}^{\varkappa}| + |u_{\varepsilon}^{\varkappa} |\|_{L_p ((-r_0^2, 0) \times \bR^d \times B_{r_0})} \le N  \||u_{\varepsilon}| + |D_x u_{\varepsilon}|\|_{L_p ((-1, 0) \times \bR^d \times B_{1})},
\end{align*}
where $N = N (d, \delta, p,  r_0, r_1)$.
Passing to the limit as $\varkappa \to 0$, we conclude that \eqref{eq5.3.10} holds with $k = 0$. In the case when $k \ge 1$, we use the method of finite-difference quotients combined with the above argument.

Next, by \eqref{eq5.3.10}  with $k = 0$, $u_{\varepsilon} \in S_{p}  ((-r_0^2, 0) \times \bR^d \times B_{r_0})$.
Then, by the interior $S_p$ estimate (see Lemma \ref{lemma 5.1}), for any $r_1 \in (0, r_0)$ and $x_0\in \bR^d$,
$$
    \||\partial_t u_{\varepsilon} - v \cdot D_x u_{\varepsilon}| + |D^2_v u_{\varepsilon}|\|_{L_{p}  (Q_{r_1} (0,x_0,0)) }
    \le N \|u\|_{L_{p}  (Q_{r_0} (0,x_0,0)) },
$$
where $N = N (d, \delta, r_1, r_{0})$.
Passing to the limit as $\varepsilon \to 0$, we  prove that $u \in S_{p, \text{loc}} ((-1, 0)\times \bR^d \times B_1)$ and  that $(P_0  + \lambda) u = 0$ a.e. in
$(-1, 0)\times \bR^d \times B_1$.

To prove the second part of the assertion $(i)$, we note that by \eqref{eq5.3.10} and the Sobolev embedding theorem, for  a.e. $t, v \in ((-r_0^2, 0)\times B_{r_0})$
and the same $\xi$,
\begin{equation}
                \label{eq5.3.14}
     \xi (t, \cdot, v) \in C^k_{0} (\bR^d),\quad k \ge 1
\end{equation}
(see Definition \ref{definition 1.1}).
Therefore, by the pointwise formula \eqref{eq1.20},
$$
(-\Delta_x)^{1/6} \xi (t, \cdot, v) \in C^k_{0} (\bR^d)
$$
is a well-defined function, and
$$
    (-\Delta_x)^{1/6} A u_{\varepsilon} (t, \cdot, v) = A (-\Delta_x)^{1/6} u_{\varepsilon} (t, \cdot, v), \quad A = \partial_t, D^2_v.
$$
Then, $(P_0 + \lambda) (-\Delta_x)^{1/6} u_{\varepsilon} = 0$ a.e. in $(-r_0^2, 0) \times \bR^d \times B_{r_0}$. As above, by using the interior $S_p$ estimate and a limiting argument, we prove the part of the assertion $(i)$ about $(-\Delta_x)^{1/6} u$.

$(ii)$
In the sequel, we follow the argument of Lemma 6.6 of \cite{DY_21}.
Let $r_0 \in (r, 1)$, and  $\zeta \in C^{\infty}_0 (\widetilde Q_{r_0})$
be a function taking values in $[0, 1]$ such that
$\zeta = 1$ on $\widetilde{Q}_{r}$.
We split $D_x u_{\varepsilon}$ as follows:
$$
	\zeta^2 D_x u_{\varepsilon} = \zeta  (\mathcal{L} u_{\varepsilon} + \text{Comm}),
$$
where
\begin{align*}
	&\mathcal{L} u_{\varepsilon} = \cR_x  (-\Delta_x)^{1/3} (\zeta (-\Delta_x)^{1/6} u_{\varepsilon}),\\
&	\text{Comm} =  \zeta D_x u_{\varepsilon} - \cR_x  (-\Delta_x)^{1/3} (\zeta (-\Delta_x)^{1/6} u_{\varepsilon}),
\end{align*}
 and $\cR_x$ is the Riesz transform.

\textit{Estimate of $\mathcal{L} u$.}
Denote $h = (-\Delta_x)^{1/6} u_{\varepsilon}$.
Then, by  the assertion $(i)$, $\zeta h \in S_p (\bR^{1+2d}_0)$ satisfies the identity
\begin{equation}
            \label{eq5.3.5}
(P_0 + \lambda) (\zeta h)  = h P_0 \zeta - 2 (a D_v \zeta) \cdot D_v h \quad \text{in} \, \, \bR^{1+2d}_0.
\end{equation}
By the $L_p$-boundedness of the Riesz transform and Theorem \ref{theorem 4.10} applied to \eqref{eq5.3.5},
\begin{equation}
                \label{eq5.3.4}
\begin{aligned}
 & \|\mathcal{L} u_{\varepsilon}\|_{ L_p (\bR^{1+2d}_0) } \le N (d, p)  \|(-\Delta_x)^{1/3} (\zeta h)\|_{ L_p (\bR^{1+2d}_0) }\\
 &\le  N (d, p, \delta) \||h P_0 \zeta| +  |(a D_v \zeta) \cdot D_v h|\|_{ L_p (\bR^{1+2d}_0) }.
 \end{aligned}
\end{equation}
Furthermore, by \eqref{eq5.3.12} and the interior gradient estimate in Lemma \ref{lemma 5.1}, we get
\begin{equation}
                \label{eq5.3.7}
    \|(a D_v \zeta) \cdot D_v h \|_{ L_p (\bR^{1+2d}_0) }
    \le N  \|h\|_{ L_p (Q_{r_0}) },
\end{equation}
where $N = N (d, \delta, p, r, r_0)$.

\textit{Commutator estimate.}
We denote
$$
	\mathcal{A} =  D_x (-\Delta_x)^{-1/6}.
$$
By Lemma \ref{lemma A.2}, this operator can be extended to $C^1_0 (\bR^d)$ functions as follows:
$$
     \mathcal{A} \phi (x)  = \text{p.v.} \int  \phi (x-y)\frac{y}{|y|^{d+5/3}}\,dy.
$$
Furthermore, by the same lemma, for any $\phi \in C^2_0 (\bR^d)$, one has $A (-\Delta_x)^{1/6} \phi \equiv D_x \phi$.
Then, since $u_{\varepsilon} (t, \cdot, v) \in C^2_0  (\bR^d)$ (see \eqref{eq5.3.14}), for a.e $(t, v) \in  (-1, 0) \times B_1$,
\begin{align*}
    &    \text{Comm} (z) = \zeta \mathcal{A} h (z) - \mathcal{A} (\zeta h) (z) \\
  & =  \text{p.v.} \int h (t, x - y, v)   \big(\zeta (t, x, v) -  \zeta (t, x-y, v)\big) \, \frac{y}{|y|^{d+5/3 }} \, dy\\
&  = \bigg(\int_{|y| \le 1}  \ldots + \int_{|y| > 1} \ldots\bigg) =: I_1 (z) + I_2 (z).
\end{align*}
By the mean-value theorem and the Minkowski inequality,
\begin{equation}
                \label{eq5.3.8}
    \|I_1\|_{ L_p (Q_r) } \le N (d, p) \|h\|_{ L_p (Q_{1, 2}) }.
\end{equation}
Next, for any $z \in Q_r$, we have
$$
    |I_2| (z) \le 2 \int_{|y| > 1} |h (t, x-y, v)| \frac{dy}{|y|^{d+2/3}}.
$$
Then, by Lemma \ref{lemma 4.7},
\begin{equation}
                \label{eq5.3.9}
    \|I_2\|_{ L_p (Q_r) } \le  N (d, p)
	   \sum_{k = 0}^{\infty}
	   2^{-2k}
	  (|h|^p)^{1/p}_{ Q_{1, 2^k}  }.
\end{equation}
Combining \eqref{eq5.3.4} - \eqref{eq5.3.9} and passing to the limit as $\varepsilon \to 0$, we prove the assertion $(ii)$.
\end{proof}

The next lemma is  about estimates  for $P_0+\lambda$-caloric functions.
\begin{lemma}
           \label{lemma 5.4}
 Let  $p\in (1,\infty)$ and $u\in \bS_{p, \text{loc}}((-4, 0)\times \bR^{d} \times B_2)$
 (or  $S_{p, \text{loc}} ((-1, 0) \times \bR^d \times B_1)$) be a function such that
$P_0 u +\lambda u = 0$
in $(-1, 0) \times \bR^d \times B_{1}$.
Then,
for any
$j \in \{0, 1\},$
$l, m \in \{0, 1, \ldots\}$, the following assertions hold.

$(i)$ For any $1/2 \le r < R \le 1$,
\begin{equation}
                \label{eq5.5.1}
    (1+\lambda) \|\partial_t^j D_x^l D_v^m u\|_{ L_p (Q_{r}) }
    \le N (d, \delta, p, j, l, m, r, R) \|u\|_{ L_p (Q_R) }.
\end{equation}
Furthermore,
\begin{equation}
                \label{eq5.5.3}
    (1+\lambda) \|\partial_t^j D_x^l D_v^m u\|_{ L_{\infty} (Q_{r}) }
    \le N (d, \delta, p, j, l, m, r, R) \|u\|_{ L_p (Q_R) }.
\end{equation}

$(ii)$ If, additionally, $u \in \bS_p ((-4, 0)\times \bR^{d} \times B_1)$, and $j+l+m\ge 1$, then
\begin{equation}
            \label{eq5.5.2}
\begin{aligned}
         &\,  \|\partial_t^j D_x^l D_v^{m} u\|_{ L_p (Q_{1/2}) }\\
&  \le N (d, \delta, p, j, l, m) \big(\||D_v u|+\lambda^{1/2}|u|\|_{L_p (Q_1)}
    + \sum_{k = 0}^{\infty}
	 2^{-2k}\big(|(-\Delta_x)^{1/6} u|^p)^{1/p}_{ Q_{1, 2^k} }\big).
\end{aligned}
\end{equation}
As in the assertion $(i)$, we may replace the left-hand side of \eqref{eq5.5.2} with
$$
    \|\partial_t^j D_x^l D_v^{m} u\|_{ L_{\infty} (Q_{1/2}) }.
$$
\end{lemma}

\begin{proof}
$(i)$
By Lemma \ref{lemma 5.3} $(i)$, $u \in S_{p, \text{loc}} ((-1, 0)\times \bR^d \times B_1)$. In the sequel, we follow the argument of Lemma 5.6 $(i)$ in \cite{DY_21}.

\textit{Case $l, j = 0.$} First, we prove that for any $r \in (1/2, R)$  and $m = \{0, 1, 2, \ldots\}$,
\begin{equation}
            \label{5.4.1}
  \lambda \|D^{m}_v u\|_{ L_p (Q_r) } +   \|D^{m+1}_v u\|_{ L_p (Q_r) }
    \le N (d, \delta, p, r, m)
    \|u\|_{ L_p (Q_R) }.
\end{equation}
To prove this, we use an induction argument. Note that \eqref{5.4.1}
with $m = 0$ follows directly from  Lemma \ref{lemma 5.1}.
In the rest of the argument, we do some formal calculations. To make the argument rigorous, one needs to use the method of finite-difference quotient.
For $m > 0$,
we fix some multi-index  $\alpha$
of order $m$.
Then, by the product rule,
\begin{equation}
			\label{5.4.2}
	(P_0 +\lambda) (D^{\alpha}_v u)
	= \sum_{  \widetilde \alpha:  \, \widetilde \alpha < \alpha, |\widetilde \alpha| = m-1 } c_{\widetilde \alpha}
	D^{\widetilde \alpha}_v D^{\alpha - \widetilde \alpha}_{x} u,
\end{equation}
where $c_{\widetilde \alpha}$ is a constant.
Next, for any $r_1 \in (r, R)$, by Lemma \ref{lemma 5.1}, we have
\begin{equation}
            \label{5.4.3}
            \begin{aligned}
 & \lambda \|D^{m}_v u\|_{ L_p (Q_r) }  +  \| D^{m+1}_v u \|_{ L_p (Q_r)}\\
   & \leq N  \| D^{m-1}_v D_x u \|_{ L_p (Q_{r_1}) }
+N  \| D^{m}_v u \|_{ L_p (Q_{r_1}) }.
\end{aligned}
\end{equation}
Observe that for any multi-index $\beta$,
\begin{equation}
            \label{5.4.4}
    (P_0 +\lambda)(D_x^{\beta} u) = 0 \quad \text{in} \, \, (-1, 0) \times \bR^d \times B_1.
\end{equation}
Then, by the induction hypothesis and Lemma \ref{lemma 5.2},
for any $r_2 \in (r_1, R)$, we have
$$
    \| D^{m-1}_v D_x u \|_{ L_p (Q_{r_1}) }
    \le N  \| D_x u\|_{ L_p (Q_{r_2})}
    \le N  \| u\|_{ L_p (Q_{R})}.
$$
This combined with \eqref{5.4.3} and the induction hypothesis implies \eqref{5.4.1}.

\textit{Case $j = 0$.} Combining \eqref{5.4.4}, \eqref{5.4.1}, and Lemma \ref{lemma 5.2}, we obtain \eqref{eq5.5.1} with $j =  0$.

\textit{Case $j = 1$.} By \eqref{5.4.2} and \eqref{5.4.4}, for any multi-indexes $\alpha \neq 0$ and $\beta$, the function
\begin{equation}
\label{eq5.5.5}
    U = D_v^{\alpha} D_x^{\beta} u
\end{equation}
satisfies the identity
\begin{equation}
\label{eq5.4.12}
    P_0 U  + \lambda U= \sum_{  \widetilde \alpha:  \, \widetilde \alpha < \alpha, |\widetilde \alpha| = |\alpha|-1 } c_{\widetilde \alpha}
	D^{\widetilde \alpha}_v D^{\alpha + \beta - \widetilde \alpha}_{x} u \quad \text{in} \, \, (-1, 0) \times \bR^d \times B_1.
\end{equation}
Then, by  Lemma \ref{lemma 5.1} and \eqref{eq5.5.1} with $j = 0$, we conclude
\begin{equation}
\label{eq5.4.13}
\begin{aligned}
   & (1+\lambda)\|\partial_t U\|_{ L_p (Q_r) } \le  (1+\lambda) \big(\|\partial_t U - v \cdot D_x U\|_{ L_p (Q_r) } + r \|D_x U\|_{ L_p (Q_r) }\big) \\
   & \le N (1+\lambda) (\|U\|_{ L_p (Q_{r_1}) } + \|D^{|\alpha| - 1}_v D^{1 +|\beta|}_x u\|_{L_p(Q_{r_1})} \\
   &\quad + \|D_x U\|_{ L_p (Q_{r})}) \le N \|u\|_{L_p (Q_R)},
    \end{aligned}
\end{equation}
    where $N = N (d, \delta, |\alpha|, |\beta|, r, R)$.
 In the case     $\alpha = 0$, the above argument yields the same bound $ (1+\lambda) \|\partial_t U\|_{ L_p (Q_r) }  \le N \|u\|_{ L_p (Q_R) } $.
 Thus, \eqref{eq5.5.1} with $j =1$ is also valid.

Next, note that the second assertion with $j = 0$ follows from \eqref{eq5.5.1} and the Sobolev embedding theorem. To prove the estimate with $j = 1$, we use  \eqref{eq5.4.12},
 \eqref{eq5.5.3} with $j = 0$, and Lemma \ref{lemma 5.1}:
\begin{align}
\label{eq5.5.4}
    &(1+\lambda)\|\partial_t U\|_{ L_{\infty} (Q_{r})  }
    \le (1+\lambda)\||D_x U| + |D^2_v U| + \lambda |U|\|_{ L_{\infty} (Q_{r})  }\notag\\
    &\le N (1+\lambda)\|u\|_{ L_{p} (Q_{R/2+1/4}) }
    \le N\|u\|_{ L_{p} (Q_{R}) }.
\end{align}

$(ii)$ 
It suffices to show the validity of the estimate
 \begin{equation}
                \label{eq5.4.9}
 \begin{aligned}
         & \|\partial_t^j D_x^l D_v^{m} u\|_{ L_p (Q_{1/2}) }\\
          &\le N (d, \delta, R, j, l, m) (\||D_x u| + |D_v u| + \lambda^{1/2} |u|\|_{ L_p (Q_{R}) }), R \in (1/2, 1],
\end{aligned}
\end{equation}
 because the  desired assertion follows from \eqref{eq5.4.9} and Lemma \ref{lemma 5.3} $(ii)$. To prove  \eqref{eq5.4.9}, we will consider four cases.

 \textit{Case 1: $l \ge 1$.} Note that by \eqref{eq5.5.1} and \eqref{5.4.4}, one has for $1/2 < r < r_1 < R$,
 $$
    \|\partial_t^j D^l_x D^m_v u\|_{ L_p (Q_{r}) } \le N \|D_x^l u\|_{ L_p (Q_{r_1}) }.
 $$
 Hence, \eqref{eq5.4.9} holds in the case $l = 1$.
 If $l \ge 2$, we use Lemma \ref{lemma 5.2}.

 \textit{Case 2: $l = 0, m \ge 1, j = 0$}. By using an induction argument (see \eqref{5.4.3}) and Lemma \ref{lemma 5.2} as in the proof of the assertion $(i)$,
 one can show that
  \begin{equation}
                \label{eq5.4.10}
    \|D^m_v u\|_{L_p(Q_r)} \le N \||D_x u| + |D_v u|\|_{L_p(Q_R)}.
 \end{equation}

 \textit{Case 3: $l = 0, m \ge 1, j = 1$.} By \eqref{eq5.4.13} with $|\alpha| = m$ and $\beta = 0$,
 $$
    \|\partial_t D^m_v u\| \le N \||D^m_v u| + |D^{m-1}_v D_x u| + |D_x u|\|_{L_p (Q_{r_1})}.
 $$
 Now \eqref{eq5.4.9} follows from \eqref{eq5.4.10} and  \eqref{eq5.4.9} with $l  = 1$ (see Case 1).

\textit{Case 4: $l = 0, m = 0, j = 1$.} Since $\partial_t u = v \cdot D_x u - a^{i j} (t) D_{v_i v_j} u-\lambda u$ in $(-1, 0) \times \bR^d \times B_1$, by using \eqref{eq5.4.10} with $m = 2$, we get
  \begin{equation}
                \label{eq5.4.11}
    \|\partial_t u\|_{L_p(Q_r)} \le N \||D_x u| + |D_v u| + \lambda |u|\|_{L_p(Q_{R_1})}, \, \, R_1 \in (R, 1).
\end{equation}
Let $R_2 \in (R_1, 1)$.
By \eqref{eq5.5.1}, we may replace the term $\lambda \|u\|_{L_p(Q_{R_1})}$ with $\lambda^{1/2} \|u\|_{ L_p (Q_{R_2}) }$ on the right-hand side of \eqref{eq5.4.11}, and, thus, \eqref{eq5.4.9} holds.

Finally, the second part of the assertion in the case when $j = 0$ follows from \eqref{eq5.5.2} and the Sobolev embedding theorem.

In the case when $j =1$, we invoke \eqref{eq5.5.5} - \eqref{eq5.4.12}. By \eqref{eq5.5.4} and the $L_{\infty} (Q_{1/2})$ estimate of $D_x U$ and $D^2_v U$ proved in the previous paragraph, we get
\begin{align*}
       & \|\partial_t U\|_{ L_{\infty} (Q_{1/2})  }\le \||D_x U| + |D^2_v U| + \lambda |U|\|_{ L_{\infty} (Q_{1/2})  }\\
       &\le N \||D_x u| + |D_v u|+\lambda^{1/2}|u|\|_{ L_{p} (Q_{ 1}) } + \lambda \|U\|_{ L_{\infty} (Q_{1/2}) }.
\end{align*}
By \eqref{eq5.5.3}, we  replace the last term with $\lambda^{1/2} \|u\|_{ L_{p} (Q_{1}) } $. The assertion is proved.
\end{proof}

 The next result follows from direct computations.
\begin{lemma}[Scaling property of $P_0$]
			\label{lemma 4.1}
Let $p\in [1,\infty]$, $T \in (-\infty, \infty]$, and
$u \in \bS_{p, \text{loc} } (\bR^{1+2d}_T)$.
For any $z_0 \in \bR^{1+2d}_T$, denote
\begin{align}
\label{eq12.3.2}
 &   	\widetilde z = (r^2 t + t_0, r^3 x + x_0 -  r^2 t v_0, r v + v_0),\quad
	\widetilde u (z) = u (\widetilde z),\\
&
    Y = \partial_t  - v \cdot D_x,
\quad    \widetilde P_0 = \partial_t
     - v \cdot D_x
     - a^{ij} ( r^2 t + t_0 ) D_{v_i v_j}\notag.
 \end{align}
Then,
$$
    Y \widetilde u (z) = r^2 (Y u) (\widetilde z),
\quad		\widetilde P_0 \widetilde u (z) = r^2 (P u) (\widetilde z).
$$
\end{lemma}

\begin{proof}[Proof of Proposition \ref{proposition 12.3}]
    Let $\widetilde u$ and $\widetilde P_0$ be the function
    and the operator from Lemma \ref{lemma 4.1} defined  with $\nu r$ in place of $r$.
Then, by the same lemma,
$$
   \widetilde P_0 \widetilde  u  + \lambda (\nu r)^2 \widetilde u = 0 \quad \text{in} \ (- 1, 0) \times \bR^d \times B_{ 1 },
$$
and for any $c > 0$, and $A = (-\Delta_x)^{1/6}$ or $D_v$,
\begin{equation}
    \label{eq12.3.1}
 \begin{aligned}
	 &(|A u|^p)^{1/p}_{Q_{\nu r, c \nu r} (z_0)}
	 = (\nu r)^{-1}	  (|A \widetilde u|^p)^{1/p}_{Q_{1, c} },\\
   &  \bigg(|A u -  (A u)_{Q_r (z_0)}|^p\bigg)^{1/p}_{Q_r (z_0)}
	= (\nu r)^{-1}
	 \bigg(|A \widetilde u - (A\widetilde u)_{ Q_{1/\nu}}|^p\bigg)_{ Q_{1/\nu}}^{1/p}.
 \end{aligned}
\end{equation}

Next,  by Lemma \ref{lemma 5.3} $( i)$, $(-\Delta_x)^{1/6}\widetilde u \in S_{p, \text{loc}} ((-1, 0)\times \bR^d \times B_1)$, and
$$
 (\widetilde P_0 + \lambda (\nu r)^2)  (-\Delta_x)^{1/6}\widetilde u   = 0 \quad  \text{a.e. in} \ (- 1, 0) \times \bR^d \times B_{  1 },
$$
and then, by Lemma  \ref{lemma 5.4} $(i)$ with $u$ replaced with $(-\Delta_x)^{1/6} \widetilde u$,
for any $\nu \ge  2$, we get
\begin{align*}
 	&\bigg(|(-\Delta_x)^{1/6} \widetilde u - ((-\Delta_x)^{1/6} \widetilde u)_{Q_{1/\nu}}|^p\bigg)^{1/p}_{ Q_{1/\nu}} \\
 & \leq  \sup_{ z_1, z_2 \in Q_{1/\nu} } |(-\Delta_x)^{1/6} \widetilde u (z_1)  - (-\Delta_x)^{1/6} \widetilde u (z_2)|\\
	 	&\leq N(d, \delta, p) \nu^{-1}  \bigg(\fint_{Q_1} |(-\Delta_x)^{1/6} \widetilde u|^p \, dz\bigg)^{1/p}.
\end{align*}
Combining this with  \eqref{eq12.3.1}, we prove the estimate for $(-\Delta_x)^{1/6} u$.

 Arguing as above  and using  Lemma   \ref{lemma 5.4} $(ii)$,
for any $\nu \ge  2$, we obtain
\begin{align*}	& \lambda^{1/2} \nu r \big(| \widetilde u - (\widetilde u)_{Q_{1/\nu}}|^p\big)^{1/p}_{ Q_{1/\nu}}
+
\big(| D_v \widetilde u - (D_v \widetilde u)_{Q_{1/\nu}}|^p\big)^{1/p}_{ Q_{1/\nu}} \\
&
\le N \nu^{-1} \bigg(\lambda^{1/2} \nu r \,  (|\widetilde u|^p)^{1/p}_{ Q_1 } +  (|D_v \widetilde u|^p)^{1/p}_{ Q_1 }
    +   \sum_{k = 0}^{\infty}
	 2^{-2k}(|(-\Delta_x)^{1/6} \widetilde u|^p)^{1/p}_{ Q_{1, 2^k} }\bigg).
\end{align*}
Dividing both sides of the above inequality by $\nu r$ and using \eqref{eq12.3.1} yield the desired estimate.
\end{proof}

\subsection{Proof of Theorem \ref{theorem 12.1}}

The following mean oscillation estimate plays a crucial  role in the proofs of Theorems \ref{theorem 12.1} and \ref{theorem 1.1}.

\begin{proposition}
                    \label{proposition 12.4}
Let $p > 1$, $r > 0$, $\nu \geq  2$ be numbers,
$T \in (-\infty, \infty]$,
 $z_0 \in \overline{\bR^{1+2d}_T}$,
and
$u \in \bS_{p } (\bR^{1+2d}_T)$,
$\vec f, g  \in L_{p} (\bR^{1+2d}_T)$  be functions such that
$$
    P_0 u +\lambda u= \div \vec f +g.
$$
Then, there exists a constant $N = N (d, \delta, p) > 0$ such that
 \begin{align*}
    I_1 &:=
     \bigg(|(-\Delta_x)^{1/6} u -  ((-\Delta_x)^{1/6} u)_{ Q_r (z_0) }|^p\bigg)^{1/p}_{ Q_r (z_0) } \\
&\leq N  \nu^{-1}  (|(-\Delta_x)^{1/6} u|^p)^{1/p}_{ Q_{\nu r } (z_0)}\\
  &  \quad + N  \nu^{(4d+2)/p}
    \sum_{k=0}^\infty 2^{-k}
	\big(  (|\vec f|^p)^{1/p}_{  Q_{2 \nu r, 2^{k+ 1} (2\nu r) } (z_0)  }
	  + \lambda^{-1/2}(|g|^p)^{1/p}_{  Q_{ 2 \nu r, 2^{k+ 1} (2\nu r) } (z_0)  }\big),
 \end{align*}
  \begin{align*}
 I_2 &:= \lambda^{1/2} \bigg(|u -  (u)_{ Q_r (z_0) }|^p\bigg)^{1/p}_{ Q_r (z_0)}
 	+ \bigg(|D_v u -  (D_v u)_{ Q_r (z_0) }|^p\bigg)^{1/p}_{ Q_r (z_0)}\\
 &
 	 \leq N   \nu^{-1} \lambda^{1/2} (|u|^p)^{1/p}_{ Q_{\nu r} (z_0)  } +
 	 N   \nu^{-1}
    (|D_v u|^p)^{1/p}_{ Q_{\nu r} (z_0)  } \\
    &\quad + N \nu^{-1}
    \sum_{k = 0}^{\infty}
	 2^{-2k}
	 (|(-\Delta_x)^{1/6} u|^p)^{1/p}_{ Q_{\nu r, 2^k \nu r} (z_0) }\\
&\quad   + N  \nu^{(4d+2)/p}
    \sum_{k=0}^\infty 2^{-k}
	\big(  (|\vec f|^p)^{1/p}_{  Q_{ 2 \nu r, 2^{k+ 1} (2\nu r) } (z_0)  }
	  + \lambda^{-1/2}(|g|^p)^{1/p}_{  Q_{2 \nu r, 2^{k+ 1} (2\nu r) } (z_0)  }\big),
  \end{align*}
\end{proposition}

\begin{proof}
\textit{Estimate of $I_1$.}
We fix some function $\phi = \phi (t, v) \in C^{\infty}_0 ((t_0 - ( 2\nu r)^2, t_0  + (2\nu r)^2) \times B_{2\nu r} (v_0))$
such that $\phi = 1$
on
$
(t_0 - (\nu r)^2, t_0) \times B_{\nu r} (v_0).
$
By Theorem \ref{theorem 11.3} $(ii)$,
the Cauchy problem (see Definition \ref{definition 2.1})
$$
    P_0 u_0  = \div (\vec f \phi) + g \phi, \quad u_0 (t_0 - (2 \nu r)^2, \cdot) \equiv 0,
$$
has a unique solution $u_0 \in \bS_p ((t_0 - (4  \nu r)^2, t_0) \times \bR^{2d})$.
To obtain a mean-oscillation estimate of
$(-\Delta_x)^{1/6} u_0$, we use the argument of Proposition \ref{proposition 12.3}.
Let $\widetilde u_0$ and $\widetilde P_0$ be the function
    and the operator from Lemma \ref{lemma 4.1} defined  with
    $ 2 \nu r$ in place of $r$ and with $u_0$  replaced with $u$. We define the functions $\widetilde  f_i, i = 1, \ldots, d$ and $\widetilde g$ by \eqref{eq12.3.2}.
Then, we have
$$
	\widetilde P_0 \widetilde u_0 + \lambda ( 2\nu r)^2 \widetilde u_0 = ( 2 \nu r) D_{v_i} \widetilde f_i + ( 2 \nu r)^2 \widetilde g \quad \text{in} \, \, (-1, 0) \times \bR^d \times B_1.
$$	
 By Lemma \ref{lemma 12.1} with
\begin{itemize}
\item[--] $(2 \nu r)^2 \lambda$ in place of $\lambda$,
\item[--]  $2 \nu r\widetilde f_i$ in place of $f_i ,i  =1, \ldots, d$,
\item[--] $(2\nu r)^2 \widetilde g$ in place of $g$,
\end{itemize}
 for any $R \ge 1$,
    \begin{align*}
        &    (|(-\Delta_x)^{1/6} \widetilde u_0|^p)^{1/p}_{Q_{1, R}}\\
   & \le N (d, p, \delta) ( 2 \nu r)  \sum_{k = 0}^{\infty} 2^{-k}   \big(\sum_{i = 1}^d(|\widetilde f_i|^p)^{1/p}_{ Q_{1, 2^{k+ 1  } R}  } + \lambda^{-1/2} (|\widetilde g|^p)^{1/p}_{ Q_{1, 2^{k+ 1  } R}  }\big) .
    \end{align*}

By dividing both sides of the above inequality by $2\nu r$ and using \eqref{eq12.3.1} with $\nu r$ replaced with $2 \nu r$, for any $R \ge 1$, we obtain
\begin{align}
    \label{12.3.1}
        (|(-\Delta_x)^{1/6} u_0|^p)^{1/p}_{Q_{\nu r, R\nu r} (z_0)}&\le N \sum_{k = 0}^{\infty} 2^{-k} F_k (R) ,\\
    \label{12.3.2}
    (|(-\Delta_x)^{1/6} u_0|^p)^{1/p}_{Q_{ r, Rr } (z_0)}
   & \le N \nu^{(4d+2)/p} \sum_{k = 0}^{\infty} 2^{-k}  F_k (R),
 \end{align}
 where
 $$
    F_k (R) = (|\vec f|^p)^{1/p}_{ Q_{2 \nu r,  2^{k+ 1  } R ( 2\nu r) } (z_0) } + \lambda^{-1/2}(|g|^p)^{1/p}_{  Q_{ 2\nu r, 2^{k+ 1} R  (2 \nu r)  } (z_0)  }.
$$
 Next, note that the function $u_h : =u - u_0 \in \bS_{p} ((t_0 - (2 \nu r)^2) \times \bR^{2d})$ satisfies
 $$
    P_0 u_h  = \div (\vec f (1-\phi)) +g (1-\phi) \quad \text{in} \, \, (t_0 - (2\nu r)^2, t_0) \times \bR^{2d}.
 $$
 Since $\vec f (1-\phi)$ and $g (1-\phi)$ vanish inside $(t_0 - ( \nu r)^2, t_0) \times \bR^d \times B_{ \nu r},$ by Proposition \ref{proposition 12.3}  and \eqref{12.3.1} with $R = 1$,
 \begin{align*}
    &    \bigg(|(-\Delta_x)^{1/6} u_h -  ((-\Delta_x)^{1/6} u_h)_{ Q_r (z_0) }|^p\bigg)^{1/p}_{Q_r (z_0)}
\leq N  \nu^{-1}
    (|(-\Delta_x)^{1/6} u_h|^p)^{1/p}_{ Q_{\nu r } (z_0)}\\
 &
    \le N  \nu^{-1}
    (|(-\Delta_x)^{1/6} u|^p)^{1/p}_{ Q_{\nu r } (z_0)}
  +  N \nu^{-1}  \sum_{k = 0}^{\infty} 2^{-k}  F_k (1).
   \end{align*}
 Finally, the mean oscillation estimate of $(-\Delta_x)^{1/6} u$ follows from the above inequality and \eqref{12.3.2} with $R = 1$.

 \textit{Estimate of $I_2$.}
 First, by Lemma \ref{lemma 12.1} and  the scaling argument presented above,
 \begin{align}
    \label{12.3.3}
     \lambda^{1/2} (|u_0|^p)^{1/p}_{Q_{\nu r} (z_0)} + (|D_v u_0|^p)^{1/p}_{Q_{\nu r} (z_0)}
    &  \le N \sum_{k = 0}^{\infty} 2^{-k^2/8} F_k (1), \\
    \label{12.3.4}
   \lambda^{1/2} (|u_0|^p)^{1/p}_{Q_{ r} (z_0)} + (|D_v u_0|^p)^{1/p}_{Q_{ r} (z_0)}
   & \le N \nu^{(4d+2)/p} \sum_{k = 0}^{\infty} 2^{-k^2/8} F_k (1).
 \end{align}
 Hence, as above, by \eqref{12.3.4}, it remains to estimate $I_2$  with $u$ replaced with $u_h$.

Next, by Proposition \ref{proposition 12.3}, we get
\begin{equation*}
\begin{aligned}
& \lambda^{1/2} \bigg(|u_h -  (u_h)_{ Q_r (z_0) }|^p\bigg)^{1/p}_{ Q_r (z_0)}
 	+
\bigg(|D_v u_h -  (D_v u_h)_{ Q_r (z_0) }|^p\bigg)^{1/p}_{ Q_r (z_0) } \\
 &
 \leq N   \nu^{-1}
    \lambda^{1/2} (|u|^p)^{1/p}_{ Q_{\nu r} (z_0)  }+
 N   \nu^{-1}
    (|D_v u|^p)^{1/p}_{ Q_{\nu r} (z_0)  }\\
    &\quad + N \nu^{-1}
    \sum_{j = 0}^{\infty}
	 2^{-2j}
	 (|(-\Delta_x)^{1/6} u|^p)^{1/p}_{ Q_{\nu r, 2^j \nu r} (z_0)}\\
&\quad + N \nu^{-1} (A_1 + A_2),
\end{aligned}
\end{equation*}
  where
  \begin{align*}
   &       A_1 =
    \sum_{j = 0}^{\infty}
	 2^{-2j}
	 (|(-\Delta_x)^{1/6} u_0|^p)^{1/p}_{ Q_{\nu r, 2^j \nu r} (z_0)}, \\
  &
	 A_2 = \lambda^{1/2} (|u_0|^p)^{1/p}_{ Q_{\nu r} (z_0)   } + (|D_v u_0|^p)^{1/p}_{ Q_{\nu r} (z_0)  }.
     \end{align*}
By \eqref{12.3.1} with $R = 2^j$, we get
  $$
    A_1 \le N \sum_{j = 0}^{\infty} 2^{-2j }   \sum_{k = 0}^{\infty} 2^{-k}
    F_k (2^j).
  $$
Using the fact that $F_k (2^j) = F_{k+j} (1)$ and changing the index of summation $k \to k+j$ yield
  $$
        A_1 \le N \sum_{j = 0}^{\infty} 2^{-j} F_j (1).
  $$
Finally, note that the term $A_2$ is estimated in \eqref{12.3.3}. The lemma is proved.
\end{proof}

\begin{proof}[Proof of Theorem \ref{theorem 12.1}]
In the first two steps below, we assume, additionally, that
\begin{equation}
			\label{eq12.1.15}
		(-\Delta_x)^{1/6} u  \in L_{p, r_1, \ldots, r_d, q} (\bR^{1+2d}_T, w).
\end{equation}
We will remove this assumption in Step 3.

\textit{Step 1: Estimate of a localized function.}
By Lemma \ref{lemma 9.7} and the self-improving property of the $A_p$-weights (see, for instance, Corollary 7.2.6 of \cite{G_14}),
there exists a number
\begin{equation*}
    p_0 = p_0 (d, p, r_1, \ldots, r_d, q,  K),\quad
    1 < p_0 < \min\{p, r_1, \ldots, r_d, q\},
\end{equation*}
such that
\begin{align}
                \label{12.1.4}
   & L_{p, r_1, \ldots, r_d, q} (\bR^{1+2d}_T, w) \subset  L_{p_0, \text{loc}} (\bR^{1+2d}_T),\\
                \label{12.1.3}
 & w_0\in A_{q/p_0} (\bR),\quad w_i\in A_{r_i/p_0} (\bR),\quad i=1,\ldots,d.
\end{align}

Let $\phi \in C^{\infty}_0 (\bR^{1+2d})$ be a function such that $\phi = 1$ on $\tQ_1$ and
denote
\begin{equation}
             \label{12.1.8}
        \phi_n  (z)  = \phi (t/n^2, x/n^3, v/n), \quad u_n  = u \phi_n, \quad \vec f_n  = \vec f \phi_n.
\end{equation}
Observe that $u_n$  satisfies
\begin{equation*}
    P_0 u_n + \lambda u_n = \div (\vec f_n)  +g_n,
\end{equation*}
where
\begin{equation}
                \label{eq12.1.20}
	g_n  = g \phi_n - \vec f \cdot D_v \phi_n  + u P_0 \phi_n - 2 (a D_v \phi_n) \cdot D_v u.
\end{equation}

Note that
$$
	\vec f_n, g_n \in L_{p_0} (\bR^{1+2d}_T),\quad u_n \in \bS_{p_0} (\bR^{1+2d}_T).
$$
We now use Proposition \ref{proposition 12.4} and conclude that
for any $z_0 \in \overline{\bR^{1+2d}_T}$,
\begin{align}
                \label{12.1.1}
   & ((-\Delta_x)^{1/6} u_n)^{\#}_T (z_0) \le N \nu^{-1}   \cM^{1/p_0}_{T} |(-\Delta_x)^{1/6} u_n|^{p_0} (z_0)\\
&     + N \nu^{(4d+2)/p_0}
    \sum_{k = 0}^{\infty} 2^{ - k }
    \big(\bM^{1/p_0}_{2^{k +1 }, T} |\vec f_n|^{p_0} (z_0)
    +  \lambda^{-1/2} \bM^{1/p_0}_{2^{k +1 }, T} | g_n|^{p_0} (z_0)\big) \notag,\\
            \label{12.1.2}
&\lambda^{1/2} (u_n)^{\#}_T (z_0)
+ (D_v u_n)^{\#}_T (z_0)\\
&\le N \nu^{-1} \lambda^{1/2} \cM^{1/p_0}_{ T}
     |u_n|^{p_0} (z_0)
+ N \nu^{-1} \cM^{1/p_0}_{ T}
     |D_v u_n|^{p_0} (z_0)\notag\\
&\quad + N \nu^{-1}   \sum_{k = 0}^{\infty} 2^{-2k} \bM^{1/p_0}_{2^k, T}  |(-\Delta_x)^{1/6} u_n|^{p_0} (z_0)\notag\\
&
\quad+ N \nu^{(4d+2)/p_0}  \sum_{k = 0}^{\infty}
2^{-  k  }
        \big(\bM^{1/p_0}_{2^{k +1 }, T} |\vec f_n|^{p_0} (z_0)
    + \lambda^{-1/2} \bM^{1/p_0}_{2^{k +1 }, T} | g_n|^{p_0} (z_0)\big)\notag,
\end{align}
where $\mathbb{M}_{c, T} f$ and $\mathcal{M}_T f$ are defined in \eqref{eq3.38}.
We take the $\|\cdot \|$-norm on both sides of \eqref{12.1.1}-\eqref{12.1.2}.
Then we  use Theorem  \ref{theorem A.4} with $p/p_0$, $q/p_0$, $r_i/p_0 > 1, i = 1, \ldots, d,$ combined with \eqref{12.1.3}.  By this and the  Minkowski inequality, we obtain
\begin{align}
            \label{12.1.5}
  &  \|(-\Delta_x)^{1/6} u_n\|
    \le N \nu^{-1} \|(-\Delta_x)^{1/6} u_n\|\\
    &+ N \nu^{(4d+2)/p_0} (\|\vec f_n\| + \lambda^{-1/2} \|g_n\|),\notag\\
			\label{12.1.9}
&
   \lambda^{1/2} \|u_n\| + \|D_v u_n\|
   \le N \nu^{-1} (\lambda^{1/2} \|u_n\|
   +  \|D_v u_n\|)\\
   & + N \nu^{-1} \|(-\Delta_x)^{1/6} u_n\|
    + N \nu^{(4d+2)/p_0}
    (\|\vec f_n\| + \lambda^{-1/2} \|g_n\|)\notag.
\end{align}
Taking $\nu \ge 2 + 4N$, we cancel the term  $\|(-\Delta_x)^{1/6} u_n\|$
on the right-hand side of \eqref{12.1.5} and  obtain
\begin{equation}
                    \label{eq7.44}
    \|(-\Delta_x)^{1/6} u_n\| \le N (\|\vec f_n\| +  \lambda^{-1/2} \|g_n\|).
\end{equation}
By  using the last inequality, \eqref{12.1.9}, and our choice of $\nu$, we prove
\begin{equation}
			\label{12.1.12}
    \lambda^{1/2} \|u_n\| +  \|D_v u_n\|
    \le N
    (\|\vec f_n\|  +   \lambda^{-1/2}\|g_n\|).
\end{equation}

\textit{Step 2: Limiting argument.}
By \eqref{12.1.12}, \eqref{eq12.1.20},  and the construction of $\phi_n$ (see \eqref{12.1.8}), we have
\begin{align*}
&\|\lambda^{1/2} |u| + |D_v u|\|_{ L_{p, r_1, \ldots, r_d, q} (\tQ_{n} \cap \bR^{1+2d}_T)  }\\
 &   \le N   \|\vec f\| +   N \lambda^{-1/2} \|g\| + N  n^{-1} \lambda^{-1/2} (\|\vec f\|
    + \|D_v u\| + \|u\|).
\end{align*}
Passing to the limit as $n \to \infty$, we prove the estimate \eqref{2.1.2} for $u$ and $D_v u$.

Next, note that due to \eqref{eq1.15}, and H\"older's inequality  for any $\eta \in C^{\infty}_0 (\bR^{1+2d}_T)$,
\begin{equation*}
	(-\Delta_x)^{1/6} \eta \in L_{*}: = L_{p^{*}, r_1^{*}, \ldots, r_d^{*}, q^{*}} (\bR^{1+2d}_T, w_{*}),
\end{equation*}
where $p^{*}, r_1^{*}, \ldots, r_d^{*}, q^{*}$ are H\"older's conjugates relative to $p, r_1, \ldots, r_d, q$ and
$$
	w_{*} (t, v) = w_0^{-1/(q-1)} (t) \prod_{i = 1}^d w_i^{ -1/(r_i-1) } (v_i).
$$	
Then, by this and the convergence $u_n \to u$   in $L_{p, r_1, \ldots, r_d, q} (\bR^{1+2d}_T, w)$, we have
$$
    	\bigg|\int_{  \bR^{1+2d}_T } ((-\Delta_x)^{1/6} u) \eta \, dz\bigg|   \le \| \eta \|_{ L_{*}  }  \nlimsup_{n \to \infty}  \| (-\Delta_x)^{1/6} u_n\|.
$$
The above inequality combined with \eqref{eq7.44}
gives \eqref{2.1.2} for $(-\Delta_x)^{1/6} u$.

\textit{Step 3: removing the assumption \eqref{eq12.1.15}.}
Let $u_{\varepsilon}$ be the convolution of $u$ in the $x$ variable with $\varepsilon^{-d} \zeta (\cdot/\varepsilon)$, where $\zeta$ is a smooth cutoff function with the unit integral.
We note that
$$
	P_0 u_{\varepsilon} + \lambda u_{\varepsilon} = \div \vec f_{\varepsilon} + g_{\varepsilon}.
$$
Furthermore,  by \eqref{eq1.15}, $(-\Delta_x)^{1/6} \zeta (\cdot/\varepsilon)$ satisfies the condition of Lemma \ref{lemma A.7}. Hence,  due to the identity
$$
	(-\Delta_x)^{1/6} u_{\varepsilon}  = \varepsilon^{-d} u \ast (-\Delta_x)^{1/6} \zeta (\cdot/\varepsilon)
$$
and Lemma \ref{lemma A.7}, the condition \eqref{eq12.1.15} holds with $u$ replaced with $u_{\varepsilon}$.
Then, by what was proved above and Lemma \ref{lemma A.7},
\begin{equation}
					\label{eq12.1.16}
\begin{aligned}				
&	\|\lambda^{1/2} |u_{\varepsilon}| + |D_v u_{\varepsilon}| + |(-\Delta_x)^{1/6} u_{\varepsilon}|\|\\
&	\le N \||\vec f_{\varepsilon}| + \lambda^{-1/2}  |g_{\varepsilon}|\|
 \le N \||\vec f| + \lambda^{-1/2}|g|\|.
\end{aligned}	
\end{equation}
By using a duality argument as in Step 2 and \eqref{eq12.1.16}, we conclude that
\eqref{eq12.1.15} and \eqref{2.1.2} holds for $\lambda > 0$.

\textit{Step 4: case $g \equiv 0$, $\lambda = 0$.}  For any $\lambda > 0$, we have
$$
	P_0 u + \lambda u = \div \vec f +  \lambda u.
$$
Then, by \eqref{2.1.2} with $\lambda > 0$, and $\lambda u$ in place of $g$,
$$
	\|\lambda^{1/2} |u| + |D_v u| + |(-\Delta_x)^{1/6} u|\|
	\le N \|f\|  + \lambda^{1/2} \|u\|.
$$
Taking the limit as $\lambda \downarrow 0$, we prove the desired bound.

To prove the assertion for the space $L_{p; r_1, \ldots, r_d} (\bR^{1+2d}_T, |x|^\alpha \prod_{i = 1}^d w_i (v_i))$,
we follow the above argument, only  modifying the proof of the estimate for  $(-\Delta_x)^{1/6} u$.
In particular, in Steps 2 - 3, due to \eqref{eq1.15},  for any $\alpha \in (-1, p-1)$ and any $\eta \in C^{\infty}_0 (\bR^{1+2d}_T)$,  one has
$$
	  (-\Delta_x)^{1/6} \eta \in L_{p^{*}; r_1^{*}, \ldots, r_d^{*}} (\bR^{1+2d}_T, |x|^{-\alpha/(p-1)} \prod_{i = 1}^d w_i^{-1/(r_i-1)} (v_i)) ,
$$
where the latter  is defined by \eqref{eq2.2.1} with
$$
p^{*},\, r_1^{*}, \ldots, r_d^{*},\, -\alpha/(p-1),\, w_1^{-1/(r_1-1)}, \ldots, w_d^{-1/(r_d-1)}
$$
in place of $p, r_1, \ldots, r_d, \alpha$, $w_1, \ldots, w_d$, respectively. The theorem is proved.
\end{proof}

\section{Proof of Theorem \ref{theorem 1.1} }
									\label{section 4}

\subsection{Proof of the assertion \texorpdfstring{$(i)$}{}}

In this section, we prove the main result for the KFP equation in the  space  $\bS_{p,  r_1, \ldots,  r_d, q} (\bR^{1+2d}_T,  w)$.
 The assertion $(iv)$ of Theorem \ref{theorem 1.1} is proved  along  the lines of this section (see Remark \ref{remark 4.1}).
We start by proving a mean oscillation estimate, which generalizes the one in Proposition \ref{proposition 12.4}.

\begin{lemma}
            \label{lemma 11.1}
Let $\lambda \ge  0$, $\gamma_0 > 0$, $\nu \ge 2$,  $p_1 \in (1, \infty)$, $\alpha \in (1,  3/2  )$  be numbers, $T \in (-\infty, \infty]$, and $R_0$ be the constant in
Assumption \ref{assumption 2.2} $(\gamma_0)$.
 Let $u \in \bS_{p_1 } (\bR^{1+2d}_T)$, $ \vec f, g \in L_{p_1} (\bR^{1+2d}_T)$ be functions such that
 \begin{equation}
            \label{eq11.6.0}
         \cP u  +\lambda u= \div \vec f+g.
  \end{equation}
 Then, under Assumptions \ref{assumption 2.1} -  \ref{assumption 2.2} $(\gamma_0)$,
there exists a sequence of positive numbers $\{c_{k  }, k   \ge  0 \}$ such  that
$$
	\sum_{k = 0 }^{\infty} c_k \leq N_0 (d, p_1, \alpha),
$$
and for  any $z_0 \in \overline{\bR^{1+2d}_T}$
and $r \in (0, R_0/(4\nu))$,
\begin{align*}
  &\lambda^{1/2} \bigg(| u -  (u)_{ Q_r (z_0) }|^{p_1}\bigg)^{1/p_1}_{ Q_r (z_0)} +
  \bigg(|D_v u -  (D_v u)_{ Q_r (z_0) }|^{p_1}\bigg)^{1/p_1}_{ Q_r (z_0)}\\
 &
 	 \leq 	N   \nu^{-1} \lambda^{1/2}
    (|u|^{p_1})^{1/p_1}_{ Q_{\nu r} (z_0)  } +
 	 N   \nu^{-1}
    (|D_v u|^{p_1})^{1/p_1}_{ Q_{\nu r} (z_0)  }\\
   &\quad + N \nu^{-1}
    \sum_{k = 0}^{\infty}
	 2^{-2k}
	 (|(-\Delta_x)^{1/6} u|^{p_1})^{1/p_1}_{ Q_{\nu r, 2^k \nu r} (z_0) }\\
	&\quad + N  \nu^{(4d+2)/p_1}
    \sum_{k=0}^\infty 2^{-k}
	 \big( (|\vec f|^{p_1})^{1/p_1}_{  Q_{ 2 \nu r, 2^{k+  1 } (2\nu r) } (z_0)  }+\lambda^{-1/2} (|g|^{p_1})^{1/p_1}_{  Q_{ 2 \nu r, 2^{k+   1 } (2\nu r) } (z_0)  }\big)\\
	  &\quad + N \nu^{(4d+2)/p_1} \gamma_0^{ (\alpha - 1)/(\alpha p_1) } \sum_{k=0}^\infty c_k (|D_v u|^{p_1 \alpha })^{1/(p_1 \alpha )}_{  Q_{ 2 \nu r, 2^{k+ 1 } (2\nu r) } (z_0) },
 \end{align*}
 \begin{align*}
    &     \bigg(|(-\Delta_x)^{1/6} u -  ((-\Delta_x)^{1/6} u)_{ Q_r (z_0) }|^{p_1}\bigg)^{1/p_1}_{ Q_r (z_0) } \\
&\leq N  \nu^{-1}  (|(-\Delta_x)^{1/6} u|^{p_1})^{1/p_1}_{ Q_{\nu r } (z_0)}\\
	&\quad + N  \nu^{(4d+2)/p_1}
    \sum_{k=0}^\infty 2^{-k}
	 \big( (|\vec f|^{p_1})^{1/p_1}_{  Q_{ 2 \nu r, 2^{k+  1 } (2\nu r) } (z_0)  }+\lambda^{-1/2} (|g|^{p_1})^{1/p_1}_{  Q_{  2 \nu r, 2^{k+  1 } (2\nu r) } (z_0)  }\big)\\
	  &\quad + N \nu^{(4d+2)/p_1} \gamma_0^{ (\alpha - 1)/(\alpha p_1) } \sum_{k=0}^\infty c_k (|D_v u|^{p_1 \alpha })^{1/(p_1 \alpha )}_{  Q_{2 \nu r, 2^{k+ 1 } (2\nu r) } (z_0) },
 \end{align*}
where $N = N (d, \delta, p_1, \alpha)$.
\end{lemma}

\begin{proof}
Clearly, we may assume that $D_v u \in L_{p_1 \alpha} ( Q_{ 2  \nu r, 2^{k+  1 } (2\nu r)} (z_0) )$ for any $k\ge 0$.
Thanks to Lemma \ref{lemma 4.1}, we may also assume that $z_0 = 0$.

We introduce
$$
\bar a (t)
 = (a (t, \cdot, \cdot))_{  B_{  (2 \nu r)^3   } \times B_{ 2\nu r} }
\quad\text{and}\quad
 \bar P_0 = \partial_t -  v \cdot D_x - \bar a^{i j} D_{v_i v_j}.
 $$
Observe that $u$ satisfies
$    \bar P_0 u  +\lambda u= \div (\vec f + (a -  \bar a) D_v u)+g$.
By this and Proposition \ref{proposition 12.4},
\begin{align*}
  &  \lambda^{1/2} \bigg(| u -  (u)_{ Q_r (z_0) }|^{p_1}\bigg)^{1/p_1}_{ Q_r } +
  \bigg(|D_v u -  (D_v u)_{ Q_r (z_0) }|^{p_1}\bigg)^{1/p_1}_{ Q_r }\\
 &
 	 \leq  N   \nu^{-1}
    (|D_v u|^{p_1})^{1/p_1}_{ Q_{\nu r}   }
    + N \nu^{-1}
    \sum_{k = 0}^{\infty}
	 2^{-2k}
	 (|(-\Delta_x)^{1/6} u|^{p_1})^{1/p_1}_{ Q_{\nu r, 2^k \nu r}  }\\
	&\quad + N  \nu^{(4d+2)/p_1}
    \sum_{k=0}^\infty 2^{-k}
	(  (|\vec f|^{p_1})^{1/p_1}_{  Q_{ 2 \nu r, 2^{k+  1 } (2\nu r) }   }+ \lambda^{-1/2} (|g|^{p_1})^{1/p_1}_{  Q_{2 \nu r, 2^{k+  1 } (2\nu r) }   }\big)\\
	  &\quad + N  \nu^{(4d+2)/p_1}
    \sum_{k=0}^\infty 2^{-k}
	  (|a - \bar a|^{p_1} |D_v u|^{p_1})^{1/p_1}_{  Q_{2 \nu r, 2^{k+ 1 } (2\nu r) }   }.
\end{align*}	
Using H\"older's inequality with $\alpha$  and $\alpha_1:= \alpha/(\alpha - 1)$ gives
\begin{align*}
& I: =(|a - \bar a|^{p_1}\, |D_v u|^{p_1})^{1/p_1}_{  Q_{2\nu r, 2^{k+1 } (2\nu r) }   }\\
&\le  (|a - \bar a|^{p_1 \alpha_1})^{1/(p_1 \alpha_1)}_{  Q_{2 \nu r, 2^{k+ 1} (2\nu r) }   }\, (|D_v u|^{p_1 \alpha})^{1/(p_1 \alpha)}_{  Q_{ 2 \nu r, 2^{k+  1 } (2\nu r) }   } =: I_1^{1/(p_1 \alpha_1)} I_2^{1/(p_1 \alpha)}.
\end{align*}
Due to the boundedness of the function $a$, we have
$$
    I_1 \le N (|a - \bar a|)_{  Q_{2 \nu r, 2^{k+ 1 } (2\nu r) }   }.
$$
Furthermore, since $ 2\nu r \le R_0/2$, by  Lemma \ref{lemma 9.10} with $c = 2^{k+1}$,
$$
    I_1 \le N 2^{3k} \gamma_0,
$$
and then,
$$
    2^{-k} I_1^{1/(p_1 \alpha_1)} \le N 2^{-k + 3k/(p_1 \alpha_1)} \gamma_0^{1/(p_1 \alpha_1)}.
$$
We set $c_k = 2^{-k + 3k/(p_1 \alpha_1)}, k \ge 0$
and note that
$
\sum_k c_k < \infty,
$ since $\alpha_1 > 3$.
The estimate for $(-\Delta_x)^{1/6} u$ is established in the same way.
The lemma is proved.
\end{proof}

In the next lemma, we prove the a priori estimate
\eqref{2.1.2} with $b \equiv \bar b \equiv 0$, $c \equiv 0$,  and  compactly supported
$u \in \bS_{p, r_1, \ldots, r_d, q} (\bR^{1+2d}_T, w),\,\vec f, g \in  L_{p, r_1, \ldots, r_d, q} (\bR^{1+2d}_T, w)$.

\begin{lemma}
            \label{lemma 11.6}
Let
\begin{itemize}[--]
\item $\lambda   > 0$, $p, r_1, \ldots, r_d, q > 1, K \ge 1$  be numbers, $T \in (-\infty, \infty]$,
\item $w_i, i = 0, 1, \ldots, d$, be weights on $\bR$ satisfying \eqref{eq1.0},
\item  Assumption \ref{assumption 2.1} be satisfied,
\item the functions  $u \in \bS_{p, r_1, \ldots, r_d, q} (\bR^{1+2d}_T, w)$,
$\vec f ,g \in L_{p, r_1, \ldots, r_d, q} (\bR^{1+2d}_T, w)$ have compact supports, and they satisfy \eqref{eq11.6.0}. 
\end{itemize}
Then, there exists a number $\gamma_0 = \gamma_0 (d, \delta,  p, r_1, \ldots, r_d, q, K) > 0$ such that, under Assumption \ref{assumption 2.2} $(\gamma_0)$, we have
\begin{equation}
			\label{eq11.6.1}
 \lambda^{1/2} \|u\| + \|D_v u\| + \|(-\Delta_x)^{1/6} u\|\le N \|\vec f\|+N \lambda^{-1/2} \|g\| +  N \lambda^{-1/2}  R_0^{-2} \|u\|,
\end{equation}
where $\|\,\cdot\,\| = \|\,\cdot\,\|_{ L_{p, r_1, \ldots, r_d, q} (\bR^{1+2d}_T, w)  }$, $N=N(d, \delta, p, r_1, \ldots, r_d, q, K)$, and
$R_0 \in (0, 1)$ is the constant in
Assumption \ref{assumption 2.2} $(\gamma_0)$.
\end{lemma}

\begin{proof}
\textit{Step 1: estimate of a function with a small support in $t$.}
Let $R_1, \gamma_0 > 0$ be numbers which we will choose later. We assume, additionally, that $u$  vanishes outside
$(s - (R_0 R_1)^2, s) \times \bR^{2d}$ for some $s \in \bR$.
The small support in time restriction  will be removed in Step 2.

Let $p_0$ be the number satisfying \eqref{12.1.4} - \eqref{12.1.3}.
Then,
\ since $u, \vec f, g$ have compact supports, we have $u , D_v u, \vec f, g \in L_{p_0} (\bR^{1+2d}_T)$ and then, by \eqref{eq11.6.0} $\partial_t u - v \cdot D_x u \in \bH^{-1}_{p_0} (\bR^{1+2d}_T)$, so that $u \in \bS_{p_0} (\bR^{1+2d}_T)$. By Corollary \ref{corollary 2.2}, $(-\Delta_x)^{1/6} u \in L_{p_0} (\bR^{1+2d}_T)$.

We fix some $\nu \ge 2$, $\alpha \in (1, \min\{3/2,p_0\})$, and denote $p_1 = p_0/\alpha$, so that $p_0 = \alpha p_1$. If  $4 \nu r \ge R_0$, then by H\"older's inequality with $\alpha$ and $\alpha_1 = \alpha/(\alpha-1)$, for any function $h \in L_{\alpha p_1, \text{loc}} (\bR^{1+2d}_T)$ and $z \in \overline{\bR^{1+2d}_T}$,
\begin{align*}
		&(|h - (h)_{Q_r (z) }|^{p_1})^{1/p_1}_{Q_r (z) }
		\le 2 (|h|^{p_1})^{1/p_1}_{Q_r (z) }\\
&		\le 2  (I_{ (s - (R_0 R_1)^2, s)  })_{Q_r (z)}^{1/(p_1\alpha_1)} (|h|^{p_1 \alpha})^{1/(p_1 \alpha)}_{Q_r (z) }\\
&		\le  2  (R_0  R_1 r^{-1})^{2/(p_1\alpha_1)}  \cM^{1/(p_1 \alpha)}_T |h|^{p_1 \alpha}(z)\\
&		\le N \nu^{2/(p_0\alpha_1)} R_1^{2/(p_1\alpha_1)}   \cM^{1/(p_1 \alpha)}_T |h|^{p_1 \alpha} (z).
\end{align*}
In the case when $4 \nu r < R_0$, we use Lemma \ref{lemma 11.1}, which is applicable, since $\vec f , g \in L_{p_1} (\bR^{1+2d}_T)$, and $u \in \bS_{p_1} (\bR^{1+2d}_T)$. Combining these cases, we get in  $\overline{\bR^{1+2d}_T}$,
 \begin{align*}
& \lambda^{1/2} (u)^{\#}_T+  (D_v u)^{\#}_T\\
	& \le N \nu^{2/(p_1\alpha_1)} R_1^{2/(p_1\alpha_1)} (\lambda^{1/2}\cM^{1/(p_1 \alpha)}_T | u|^{p_1 \alpha}+\cM^{1/(p_1 \alpha)}_T |D_v  u|^{p_1 \alpha}) \\
	&\quad + N \nu^{-1}(\lambda^{1/2} \cM^{1/p_1}_{T} |u|^{p_1}+  \cM^{1/p_1}_{T} |D_v u|^{p_1})  \\
	&\quad + N  \nu^{-1} \sum_{k = 0}^{\infty}
	  2^{-2k} \bM^{1/p_1}_{2^k, T} |(-\Delta_x)^{1/6} u|^{p_1}\\
	&\quad + N \nu^{(4d+2)/p_1}   \gamma_0^{ 1/(p_1 \alpha_1)}  \sum_{k = 0}^{\infty} c_{k  }
	\bM^{1/(p_1 \alpha)}_{2^{k +1}  , T}  |D_v u|^{p_1 \alpha}\\
	&\quad + N  \nu^{(4d+2)/p_1}    \sum_{k = 0}^{\infty} 2^{-k  }    %
   ( \bM^{1/p_1}_{2^{k +1  } , T} |\vec f|^{p_1} + \lambda^{-1/2}\bM^{1/p_1}_{2^{k +1  } , T} |g|^{p_1}),
   \end{align*}
and
 \begin{align*}
&   ((-\Delta_x)^{1/6} u)^{\#}_T\\
	& \le N \nu^{2/(p_1\alpha_1)} R_1^{2/(p_1\alpha_1)} \cM^{1/(p_1 \alpha)}_T |(-\Delta_x)^{1/6} u|^{p_1 \alpha} \\
		& \quad + N \nu^{-1}  \cM^{1/p_1}_{T} |(-\Delta_x)^{1/6} u|^{p_1}  \\
	&\quad + N \nu^{(4d+2)/p_1}   \gamma_0^{ 1/(p_1 \alpha_1)}  \sum_{k = 0}^{\infty} c_{k  } \bM^{1/(p_1 \alpha)}_{2^{k +1}  , T}  |D_v u|^{p_1 \alpha}\\
	&\quad + N  \nu^{(4d+2)/p_1}    \sum_{k = 0}^{\infty} 2^{-k  }    %
   ( \bM^{1/p_1}_{2^{k +1  } , T} |\vec f|^{p_1}+ \lambda^{-1/2}\bM^{1/p_1}_{2^{k +1  } , T} |g|^{p_1}),
   \end{align*}
where $(f)^{\#}_T$, $\bM_{c, T} f$, and $\cM_T f$ are defined in \eqref{eq3.38}.
We take the $\|\cdot \|$-norm of both sides of the above inequalities and use the  Minkowski inequality. Then, by  \eqref{12.1.3} with $p_0 = p_1 \alpha$ and   Theorem \ref{theorem A.4}
with
$$
    p/(p_1 \alpha), r_1/(p_1 \alpha), \ldots r_d/(p_1 \alpha), q/(p_1 \alpha) > 1,
$$
we obtain
\begin{align}
            \label{11.2.1}
&\lambda^{1/2}\|u\| +  \|D_v u\| \le N (\nu^{-1} +  \nu^{2/(p_1\alpha_1)} R_1^{2/(p_1\alpha_1)}) (\lambda^{1/2}\|u\| +  \|D_v u\|)\\
  &\quad
   + N \nu^{-1} \|(-\Delta_x)^{1/6} u\|
    +N\nu^{(4d+2)/p_1}   \gamma_0^{ 1/(p_1 \alpha_1)}\|D_v u\| \notag\\
    &
\quad + N \nu^{(4d+2)/p_1}( \|\vec f\| + \lambda^{-1/2}\|g\|)\notag,\\
& \|(-\Delta_x)^{1/6} u\| \le N  (\nu^{-1} +  \nu^{2/(p_1\alpha_1)} R_1^{2/(p_1\alpha_1)}) \|(-\Delta_x)^{1/6} u\|  \label{eq2.38b}\\
&\quad  + N \nu^{(4d+2)/p_1}   \gamma_0^{ 1/(p_1 \alpha_1)} \|D_v u\| + N  \nu^{(4d+2)/p_1} (\|\vec f\|+ \lambda^{-1/2}\|g\|).\notag
\end{align}
Taking $\nu \ge   2 + 4N$ first, and then choosing $R_1,\gamma_0 > 0$ sufficiently small such that
$$
 	N\nu^{(4d+2)/p_1}   \gamma_0^{ 1/(p_1 \alpha_1)} + N \nu^{2/(p_1\alpha_1)} R_1^{2/(p_1\alpha_1)} < 1/4,
$$
and using the fact that $(-\Delta_x)^{1/6} u \in L_{p, r_1, \ldots, r_d, q} (\bR^{1+2d}_T, w)$ (see Corollary \ref{corollary 3.2}), we obtain from \eqref{eq2.38b} that
\begin{equation}
                        		\label{11.2.11}
	\|(-\Delta_x)^{1/6} u\| \le  (1/2)\|D_v u\| + N  \nu^{(4d+2)/p_1} (\|\vec f\|+ \lambda^{-1/2}\|g\|).
\end{equation}
By this, \eqref{11.2.1}, and our choice of $\nu, R_1$, and $\gamma_0$, we get
\begin{equation*}
\lambda^{1/2}\|u\| + 	\|D_v u\|
 \le  (5/8) (\lambda^{1/2}\|u\| +  \|D_v u\|)
 + N   \nu^{(4d+2)/p_1} (\|\vec f\|+ \lambda^{-1/2}\|g\|),
\end{equation*}
which implies
	\begin{equation}
				\label{11.2.6}
	\lambda^{1/2}\|u\| + 	\|D_v u\|  \le  N ( \|\vec f\|+ \lambda^{-1/2}\|g\|).
	\end{equation}
This combined with \eqref{11.2.11} gives
	\begin{equation}
				\label{11.2.7}
	 \|(-\Delta_x)^{1/6} u\|  \le  N  (\|\vec f\|+ \lambda^{-1/2}\|g\|).
		\end{equation}

\textit{Step 2: partition of unity.}
Let $\zeta \in C^{\infty}_0 ((-(R_0 R_1)^2, 0))$ be a nonnegative function such that
\begin{equation}
				\label{11.2.2}
	\int \zeta^q (t) \, dt = 1, \quad |\zeta'| \le N_0 (R_0 R_1)^{-2-2/q}.
\end{equation}
Observe that for any $t \in \bR_T$
and $U \in L_{p, r_1, \ldots, r_d, q} (\bR^{1+2d}_T,w)$, by \eqref{11.2.2},
\begin{align*}
&	\|U (t, \cdot)\|^q_{  L_{p, r_1, \ldots, r_d} (\bR^{2d}, \Pi_{i=1}^d w_i) }\\
&=  \int_{\bR} \|U (t, \cdot)\|^q_{  L_{p, r_1, \ldots, r_d} (\bR^{2d}, \Pi_{i=1}^d w_i) } \zeta^q (t-s) \, ds.
\end{align*}
Multiplying the above identity by $w_0$ and integrating over $(-\infty, T]$, we get 
\begin{equation}
                \label{11.2.3}
    		\|U\|^q   = \int_{\bR} \| U   \zeta (\cdot-s)\|^q  \, ds.
\end{equation}

Next, note that for any $s \in \bR$, the function
$u_s (z) := u (z)  \zeta (t-s)$ vanishes outside $(s - (R_0 R_1)^2, s)$  and  satisfies the equation
$$
	\cP u_s (z) + \lambda u_s (z) =  \div (\vec f  (z) \zeta (t-s))
	+ g (z)  \zeta (t-s)+ u \zeta'(t -s).
$$
By  \eqref{11.2.6} and \eqref{11.2.7} proved in Step 1,
\begin{align*}
	&\lambda^{1/2} \|u_s\| + \|D_v u_s\| + \|(-\Delta_x)^{1/6} u_s\|\\
	& \le N \|\vec f \zeta (\cdot-s)\| + N \lambda^{-1/2} \|g \zeta (\cdot-s)\| + N (R_0 R_1)^{-2-2/q} \lambda^{-1/2} \|u \xi (\cdot-s)\|,
\end{align*}
where $\xi \in C^{\infty}_0 (\bR)$ is a nonnegative function such that $\xi = 1$ on the support of $\zeta$, and  $\int \xi^q (t) \, dt = N_1 (R_0 R_1)^{2}$. Raising the above inequality to the power $q$, integrating over $s \in \bR$ and using \eqref{11.2.3}, and our choice of $R_1$, we prove \eqref{eq11.6.1}.
\end{proof}

\begin{remark}
			\label{remark 4.1}
A version of Lemma  \ref{lemma 11.6} holds in the case when
\begin{align*}
	&u \in \bS_{p; r_1, \ldots, r_d} (\bR^{1+2d}_T,  |x|^{\alpha} \prod_{i = 1}^d w_i (v_i)), \\
	&\vec f , g \in L_{p; r_1, \ldots, r_d} (\bR^{1+2d}_T,  |x|^{\alpha} \prod_{i = 1}^d w_i (v_i)),
\end{align*}
and these functions have  compact supports.
To prove the result, one needs to follow the proof of Lemma \ref{lemma 11.6}, but use the partition of unity in $v_d$, instead of $t$.
We give a few details below.
First, repeating the argument of Step 1, we prove the estimate \eqref{eq11.6.1} for $u$ satisfying Eq. \eqref{eq11.6.0} and vanishing outside $\bR_t \times \bR^d_x \times \bR^{d-1}_{v_1, \ldots, v_d} \times (s- R_0 R_1, s)$ for some $s \in \bR$.
Furthermore, we  fix a nonnegative function $\zeta \in C^{\infty}_0 ((-R_0 R_1, 0))$ such that
$$
	\int \zeta^{r_d} \, dv_d = 1,\quad |\zeta'| < N (R_0 R_1)^{-1-1/r_d}
$$
and denote $u_s (z) = u (z) \zeta (v_d - s)$.
The function $u_s$ satisfies the identity
\begin{align*}
	\mathcal{P} u_s (z) + \lambda u_s (z)  &= D_{v_i}  [ f^i (z) \zeta (v_d - s) - a^{i d} (z) \zeta' (v_d - s) u (z)] \\
	&\quad - f_d (z) \zeta' (v_d - s)
	+ g \zeta (v_d - s)
	- a^{ d j} (z) \zeta' (v_d - s) D_{v_{j}} u (z).
\end{align*}
As in the proof of Lemma \ref{lemma 11.6}, we obtain
\begin{align*}
	  &  \lambda^{1/2} \|u\| + \|D_v u\| + \|(-\Delta_x)^{1/6} u\| \\
	& \le N \||\vec f| +  R_0^{-1} |u| +\lambda^{-1/2}|g|\| +  N \lambda^{-1/2}  R_0^{-1}\||D_v u| +|f_d|\|.
\end{align*}
Taking $\lambda$ sufficiently large, we may erase the  terms involving $u$ from the right-hand side of the above inequality.
\end{remark}

\begin{proof}[Proof of  Theorem \ref{theorem 1.1} $(i)$]
First, we consider the case when $b \equiv \overline{ b} \equiv 0$ and $c \equiv 0$.  We will focus on the case when the weight is independent of the $x$ variable, since in the remaining case, the proof is the same. Let $\phi_n, u_n$ be the functions defined by \eqref{12.1.8}. Note that
\begin{equation*}
  \cP u_n + \lambda u_n = \div \mathsf{f}_n + \mathsf{g_n},
\end{equation*}
where
\begin{align*}
&	\mathsf{f}_n =  \vec f \phi_n - (a D_v \phi_n) u, \\
&
	\mathsf{g}_n = g \phi_n  - \vec f \cdot D_v \phi_n+ u (\partial_t \phi_n   - v \cdot D_x \phi_n)  -  (a D_v \phi_{ n } ) \cdot D_v u,
\end{align*}
and, furthermore, the functions $\mathsf{f_n}, \mathsf{g_n}, u_n$ are compactly supported, and
\begin{equation*}
	\mathsf{f}_n, \mathsf{g}_n, u_n  \in  L_{p, r_1, \ldots, r_d, q} (\bR^{1+2d}_T, w).
\end{equation*}

Then, by Lemma \ref{lemma 11.6},
there exist
$
	\gamma_0 = \gamma_0 (d, \delta, p, r_1, \ldots, r_d, q, K) > 0
$
 such that if Assumption \ref{assumption 2.2} $(\gamma_0)$ holds, then for any $\lambda > 0$,
\begin{align*}
&	 \|\lambda^{1/2}|u_n| + |D_v u_n| + |(-\Delta_x)^{1/6} u_n|\|\\
&	 \le N \|\mathsf{f}_n\| + N \lambda^{-1/2}  \||\mathsf{g}_n| + R_0^{-2}  |u_n|\|,\\
&
\le   \|\vec f\| + N \lambda^{-1/2} (\|g\| + R_0^{-2} \|u\|)
 + N n^{-1} (1 +\lambda^{-1/2})   \|u\| \\
 &\quad + N  n^{-1} \lambda^{-1/2} \||D_v u| + |\vec f|\|,
\end{align*}
where $N = N (d, \delta, p, r_1, \ldots, r_d, q, K)$, and $R_0 \in (0, 1)$ is the number in Assumption \ref{assumption 2.2} $(\gamma_0)$.
Using a limiting argument as in Step 2 of the proof of Theorem \ref{theorem 12.1}, we conclude
$$
	 \|\lambda^{1/2}|u| + |D_v u| + |(-\Delta_x)^{1/6} u|\|
\le  N  \|f\| + N \lambda^{-1/2} \|g\| +  N \lambda^{-1/2} R_0^{-2} \|u\|.
$$
Taking $\lambda > 2  N R_0^{-2}$ so that
$\lambda^{1/2} - N \lambda^{-1/2} R_0^{-2} > \lambda^{1/2}/2$, we prove \eqref{2.1.2}.

In the general case, we rewrite Eq. \eqref{1.1} as
$$
    \cP u + \lambda u = G + \div  F, \quad
    G = g  - b \cdot D_v u  - c u, \quad F = \vec  f  - \overline{ b} u.
$$
Then, by  \eqref{2.1.2} with $F$ and $G$ in place of $ \vec f$ and $g$, we get
\begin{align*}
     \lambda^{1/2} \|u\| &+ \|D_v u\|
    +  \|(-\Delta_x)^{1/6} u\| \\
    &
    	\le N \lambda^{-1/2} \|g\|  + N L \lambda^{-1/2} (\|D_v u\| + \|u\|)
    	+ N  \|\vec f\| + N L  \|u\|.
\end{align*}
Taking $\lambda \ge  1+4 (N L)^2$, we may drop the terms involving $u$ on the right-hand side of the above inequality. The assertion $(i)$ is proved.
\end{proof}

\subsection{Proof of  Theorem \ref{theorem 1.1}  \texorpdfstring{$(ii)$}{} and \texorpdfstring{$(iii)$}{}.}
As we pointed out in Remark \ref{remark 2.2}, the assertion $(iii)$ follows directly from $(ii)$.
To prove the latter, we first establish the unique solvability result in  $L_p (\bR^{1+2d})$ spaces.

\begin{proposition}
                \label{proposition 11.6}
        Theorem \ref{theorem 1.1} $(ii)$ is satisfied in the case when $p = r_1 = \cdots = r_d = q$,  $w \equiv 1$, and $T = \infty$.
\end{proposition}

\begin{proof}
The assertion follows from the method of continuity, Theorem \ref{theorem 1.1} $(i)$, and Theorem \ref{theorem 11.3} $(i)$.
\end{proof}

The following is  a decay estimate for the solution to Eq. \eqref{1.1} with the compactly supported right-hand side, which is analogous to Lemma \ref{lemma 12.1}. This result is needed for establishing the existence part in Theorem \ref{theorem 1.1} $(ii)$.

\begin{lemma}
            \label{proposition 11.7}
Invoke the assumptions of Proposition \ref{proposition 11.6} and let $\lambda_0 =  \lambda_0 (d, \delta, p, L) > 1$  be the constant from the statement of this result.  Assume, additionally, that  $\vec f$ and $g$  vanish outside $\tQ_R$ for some  $ R \ge 1$ and let $u \in \bS_p (\bR^{1+2d})$ be the  solution to Eq. \eqref{1.1}, which exists and is unique due to the aforementioned proposition.
Then, for any $\lambda \ge \lambda_0$  and $j \in \{0, 1, 2, \ldots\}$,
\begin{align*}
&	\lambda^{1/2}  \|u\|_{ L_{p} (\tQ_{2^{j+1} R} \setminus  \tQ_{2^{j} R})  }
	+	 \|D_v u\|_{ L_{p} (
\tQ_{2^{j+1} R} \setminus  \tQ_{2^{j} R})  }\\
&	\le N 2^{-j (j-1)/4} R^{-j} \big(\| \vec f\|_{ L_{p} (\bR^{1+2d}) } + \lambda^{-1/2}\|g\|_{ L_{p} (\bR^{1+2d}) }),
\end{align*}
where $N = N (d, \delta, p, L)$.
\end{lemma}

\begin{proof}
The proof is similar to that of Lemma 7.4 in \cite{DY_21}.
First, by Theorem \ref{theorem 1.1} $(i)$, we have
\begin{equation}
                \label{11.7.1}
	\lambda^{1/2} \|u\|_{  L_p (\bR^{1+2d})
	}
    + \|D_v u\|_{  L_p (\bR^{1+2d})}
        \le N \|\vec f\|_{ L_p (\bR^{1+2d}) }  + N \lambda^{-1/2} \|g\|_{ L_p (\bR^{1+2d}) }.
\end{equation}
Let $\eta_j ,j \in \{0,1,2,\ldots\},$ be a sequence of smooth functions such that $\eta_j = 0$ in $\tQ_{2^{j} R}$, $\eta_j  = 1$ outside $\tQ_{2^{j+1} R}$,
\begin{equation}
\begin{aligned}
                \label{eq4.3.2}
&|\eta_j|\le 1,\quad |D_v \eta_j|\le N2^{-j}R^{-1}, \quad |D^2_v \eta_j|\le N2^{-2j}R^{-2},\\
	& |D_x \eta_j|\le N2^{-3j}R^{-3}, \quad |\partial_t \eta_j| \le N2^{-2j}R^{-2}.
\end{aligned}	
\end{equation}
Note that  $u_j = u \eta_j$ satisfies the  equation
\begin{align*}
      &  \cP u_{j} +  \div (\overline{ b} u_j) + b^i D_{v_i} u_j + c u_j +   \lambda u_j
    =   \div [-u \,  (a  D_v \eta_j)]  \\
    & -  (a D_v \eta_j)\cdot D_v u + u (\partial_t \eta_j - v \cdot D_x \eta_j + b \cdot D_{v} \eta_j + \overline{b} \cdot D_v \eta_j)
 \end{align*}
because $\vec f $ and $g$ vanish outside $\tQ_R$.
Then, by the a priori estimate in Theorem \ref{theorem 1.1} $(i)$, \eqref{eq4.3.2},
and the fact that $\lambda > 1$, we get
\begin{align*}
&	\|\lambda^{1/2} |u|+|D_v u|\|_{  L_{p}  (\tQ_{2^{j+2} R} \setminus  \tQ_{2^{j+1} R}) }\\
&	\le   N  \|u \,  (a  D_v \eta_j)\|_{  L_{p}  (\bR^{1+2d})   }
    + N \lambda^{-1/2}  \|(a D_v \eta_j)\cdot D_v u\|_{  L_{p}  (\bR^{1+2d})   } \\
&\quad + N \lambda^{-1/2}  \| u (\partial_t \eta_j - v \cdot D_x \eta_j + (b +\overline{b}) \cdot D_{v} \eta_j)  \|_{  L_{p}  (\bR^{1+2d})   } \\
& 	\le  N 2^{-j} R^{-1} \|\lambda^{1/2} |u| + |D_v u|\|_{  L_{p} (\tQ_{2^{j+1} R} \setminus  \tQ_{2^{j} R}) }.
\end{align*} 
Iterating the above estimate and using \eqref{11.7.1}, we prove the assertion.
\end{proof}

\begin{proof}[Proof of Theorem \ref{theorem 1.1} $(ii)$]
The uniqueness follows from Theorem \ref{theorem 1.1} $(i)$.
To prove the existence part, we follow the proof of Theorem 2.5 of \cite{DY_21}. Below we delineate the argument.

First, we consider the case $T = \infty$. By using the reverse H\"older inequality for the $A_p$-weights and the scaling property of the $A_p$-weights (see, for instance, Chapter 7 in \cite{G_14}), one can show that there exists a sufficiently large number $p_1  = p_1 (d, p, r_1, \ldots, r_d, q, K) \in (1, \infty)$ such that,
for any $h \in L_{p_1, \text{loc}} (\bR^{1+2d})$, one has
\begin{equation}
                \label{eq1.1}
    \|h\|_{ L_{p, r_1, \ldots, r_d, q} ( \tQ_R, w) }
    \le
    N R^{\kappa} \|h\|_{ L_{p_1} (  \tQ_R) },
\end{equation}
where $\kappa, N >0$ are independent of $R$ and $h$.
In addition, the above inequality also holds with  $\tQ_{2R}\setminus \tQ_R$ in place of  $\tQ_R$.

Next, let $\vec f_n$, $g_n, n \ge 1,$ be  sequences of $C^{\infty}_0 (\bR^{1+2d})$ functions converging to $\vec f$ and $g$ in $L_{p, r_1, \ldots, r_d, q} (\bR^{1+2d}, w)$, respectively.
Then, by Proposition \ref{proposition 11.6}, for any $n$, the equation
\begin{equation}
                \label{eq1.2}
    \cP u_n + \div (\overline{b} u_n) + b \cdot D_v u_n + (c+\lambda) u_n = \div \vec f_n + g_n
\end{equation}
has a unique solution $u_n \in \bS_{p_1} (\bR^{1+2d})$. Fix  any $n$ and let $R=R_n \ge 1$ be a constant such that
$\vec f_n$ and $g_n$ vanish outside $\tQ_R$.
Then, by \eqref{eq1.1} combined with Lemma \ref{proposition 11.7}, for any $j \in \{0, 1, 2, \ldots\}$,
\begin{align*}
     &   \|\lambda^{1/2} |u_n| + |D_v u_n|\|_{  L_{p, r_1, \ldots, r_d, q} (\tQ_{2^{j+1} R} \setminus \tQ_{2^{j} R}, w) }\\
&
    \le  N (2^{j} R)^{\kappa} \|\lambda^{1/2} |u_n| + |D_v u_n|\|_{ L_{p_1} (\tQ_{2^{j+1} R} \setminus \tQ_{2^{j} R}) }\\
 &
    \le N (2^{j} R)^{\kappa} 2^{-j(j-1)/4}R^{-j}
    \big(\|\vec f_n\|_{ L_{p_1} (\bR^{1+2d}) }+ \lambda^{-1/2} \|g_n\|_{ L_{p_1} (\bR^{1+2d})}\big).
 \end{align*}
The above inequality implies that $u_n \in \bS_{p, r_1, \ldots, r_d, q} (\bR^{1+2d}, w)$. Hence, by Theorem \ref{theorem 1.1} $(i)$,  $u_n, n \ge 1$
is a Cauchy sequence in $\bS_{p, r_1, \ldots, r_d, q} (\bR^{1+2d}, w)$ and has a limit $u$. Passing to the limit in \eqref{eq1.2}, we conclude the existence of the unique solution to Eq. \eqref{1.1}.

The  case $T < \infty$ is treated as in the proof of Theorem \ref{theorem 11.3} $(i)$ (see page \pageref{T finite}).
\end{proof}

\section{Proof of  Theorem \ref{theorem 1.10}}
In the next two lemmas, we prove energy identities for the operator
$$
	Y u : = \partial_t u  + \alpha (\sfv)  \cdot D_x u.
$$
For $T \in (-\infty, \infty]$,  let $H^1_2 (\bR^{1+d+d_1}_T)$ be the space of functions $u \in L_2 (\bR^{1+d+d_1}_T)$  such that $D_{\sfv} u \in L_2 (\bR^{1+d+d_1}_T)$ and let   $\langle\cdot, \cdot\rangle_T$ be the duality pairing between $\bH^{-1}_2 (\bR^{1+d+d_1}_T)$ and $H^1_2 (\bR^{1+d+d_1}_T)$ given by
\begin{equation}
			\label{eq5.1}
	\langle f, g \rangle_{T} = \int_{-\infty}^T \int_{\bR^{d }} [f (t, x, \cdot), g (t, x, \cdot)] \, dx dt,
\end{equation}
where
$$
	[f, g] = \int_{\bR^{d_1}} \big((1- \Delta_{\sfv})^{-1/2} f\big) \,  \big((1- \Delta_{\sfv})^{1/2} g\big) \, d{\sfv}.
$$

\begin{lemma}
			\label{lemma 7.1}
Let    $u \in H^1_2 (\bR^{1+d+d_1})$ be a function such that $Y u \in \bH^{-1}_2 (\bR^{1+d+d_1})$.
Then,
$$
	\langle Y u,   u \rangle_{\infty} = 0.
$$
\end{lemma}

\begin{proof}
For a distribution $h$ on  $\bR^{1+d+d_1}$, a cutoff function $\eta \in C^{\infty}_0 (\bR^{1+d+d_1})$ with the unit integral, and $\varepsilon > 0$,  we denote
\begin{equation}
			\label{eq7.1.1}
	h_\varepsilon  (t, x, \sfv)  =  \varepsilon^{-(1+\theta d/2+d_1)} (h, \eta ((t-\cdot)/\varepsilon, (x - \cdot)/\varepsilon^{\theta/2}, (\sfv-\cdot)/\varepsilon),
\end{equation}
where $(h, \eta)$ is the action of $h$ on $\eta$.
For the sake of convenience, we omit $\bR^{1+d+d_1}$ in the notation of functional spaces and write $\langle\cdot,\cdot\rangle=\langle\cdot,\cdot\rangle_\infty$.
First, we split $\langle Y u,    u\rangle$ as follows:
\begin{align*}
\langle Y u,    u\rangle &= \langle Y u_{\varepsilon},  u_{\varepsilon} \rangle+ \langle Y u - (Y u)_{\varepsilon},   u\rangle\\
&\quad	+  \langle  (Y u)_{\varepsilon} - Y u_{\varepsilon},   u\rangle+ \langle  Y u_{\varepsilon},   u -  u_{\varepsilon} \rangle\\
&  =: I_1 + I_2 + I_3 + I_4.
 \end{align*}
Since $u_{\varepsilon}$ is a smooth function vanishing at infinity,  one has $I_1 = 0$. Furthermore,   by the properties of Bessel potential spaces (see, for example, Theorem  13.9.2 in \cite{Kr_08}),
$$
	|I_2| \le  \|Y u - (Y u)_{\varepsilon}\|_{  \mathbb{H}^{-1}_2  }   \|u\|_{  H^1_2 }   \to 0  \quad \text{as} \,\, \varepsilon \to 0.
$$
Next, note that
\begin{align*}
&\big((Y u)_{\varepsilon} - Y u_{\varepsilon}\big) (t, x, \sfv)\\
&
 = \varepsilon^{-\theta/2} \int (\alpha (\sfv -\varepsilon \sfv')  -  \alpha (\sfv)) \cdot D_x \eta (t', x', \sfv')
 u(t-\varepsilon t', x - \varepsilon^{\theta/2} x', \sfv -\varepsilon \sfv')    \, dx'd\sfv'dt',
\end{align*}
and, then, by the Minkowski inequality and Assumption \ref{assumption 1.3.2}, we have
\begin{equation}
			\label{eq7.1.2}
	\|(Y u)_{\varepsilon} - Y u_{\varepsilon}\|_{ L_2  } \le \varepsilon^{\theta/2}\|u\|_{ L_2  }  \to 0
\end{equation}
as $\varepsilon \to 0$, which gives
$$
	I_3 \to 0   \quad \text{as} \,\, \varepsilon \to 0.
$$

Next, by duality,
$$
	I_4 \le  \|Y u_{\varepsilon}\|_{  \mathbb{H}^{-1}_2 }  \| u -   u_{\varepsilon}  \|_{  H^1_2  }.
$$
By \eqref{eq7.1.2}, for sufficiently small $\varepsilon > 0$,  the first  factor on the right-hand side is bounded by
\begin{align*}
\|(Y u)_{\varepsilon}\|_{  \mathbb{H}^{-1}_2 }
+  \|(Y u)_{\varepsilon} - Y u_{\varepsilon}\|_{\mathbb{H}^{-1}_2 } 
\le \|Y u\|_{  \mathbb{H}^{-1}_2 } + \|u\|_{L_2}.
\end{align*}
Thus, by this and the fact that
$$
		 \| u -   u_{\varepsilon}  \|_{  H^1_2  }  \to 0    \quad \text{as} \,\, \varepsilon \to 0,
$$
we conclude that $I_4 \to 0$ as $\varepsilon \to 0$.
The lemma is proved.
\end{proof}

\begin{lemma}
			\label{lemma 7.2}
Let   $T \in \bR$,   $u \in H^1_2 (\bR^{1+d+d_1}_{  T  })$, and $Y u \in \bH^{-1}_2 (\bR^{1+d+d_1}_{  T  })$.
Then, for a.e. $s \in  (-\infty, T]$,
$$
	\langle Y u,   u \rangle_s =  (1/2)\|u (s, \cdot)\|^2_{  L_2 (\bR^{d+d_1})},
$$
where $\langle \cdot, \cdot \rangle_s $  is defined in \eqref{eq5.1}.
\end{lemma}

\begin{proof}
We extend $u$ by $0$ for $t > T$.
Let $\xi$ be a smooth function on $\bR$ defined by
$$
    \begin{cases}
    \xi (t)  = 0,\quad \quad t \le 1,\\
     \xi (t) \in (0, 1), \, \, t \in (1, 2), \\
    \xi (t) = 1, \quad  \quad t \ge 2.
    \end{cases}
$$
We fix some $s \in  (-\infty, T]$
and denote
$\xi_{\varepsilon} (\cdot) = \xi ((s-\cdot)/\varepsilon)$,
$u_{\varepsilon} = u \xi_{\varepsilon}.$
It follows that $u_{\varepsilon} \in H^1_2 (\bR^{1+d+d_1})$ and $Y u_{\varepsilon} = \xi_{\varepsilon} (Y u) + u \xi_{\varepsilon}' \in \bH^{-1}_2 (\bR^{1+d+d_1})$.
Then,  by  Lemma \ref{lemma 7.1},
$$
	\langle Y u_{\varepsilon},  u_{\varepsilon}  \rangle_s  = 0,
$$
which gives
\begin{equation}
			\label{B.5.1}
 	 	\langle (Y u) \xi_{\varepsilon},  u_{\varepsilon}  \rangle_s  = - (1/2) \int_{\bR^{1+d+d_1}_s}  u^2 (\xi^2_{\varepsilon})' \, dx d\sfv dt.
\end{equation}
The integral on the left-hand side of \eqref{B.5.1} equals
$$
	 \int_{ \bR^{1+d}_s}  [(Y u) (t, x, \cdot),  u (t, x, \cdot) ] \xi^2_{\varepsilon} (t) \, dxdt \to
\int_{\bR^{1+d}_s} [(Y u) (t, x, \cdot), u (t, x, \cdot)]  \, dxdt
$$
as $\varepsilon \to 0$ by the dominated convergence theorem.

Note that $\int_{-\infty}^s (\xi^2_{\varepsilon})' \, dt = -1$.
Then, the right-hand side of \eqref{B.5.1}  is equal to
$$
 	\frac 1 2 \| u (s, \cdot, \cdot)\|^2_{  L_2 (\bR^{d+d_1})  }  -  \frac 1 2 \int_{ -\infty }^s  (\|  u (t, \cdot, \cdot)\|^2_{  L_2 (\bR^{d+d_1})} - \|  u (s, \cdot,\cdot)\|^2_{  L_2 (\bR^{d+d_1})})  (\xi^2_{\varepsilon} (t))' \, dt.
$$
The last term is bounded by
$$
	N \varepsilon^{-1}  \int_{ s-2\varepsilon}^{s-\varepsilon}  \big|\|  u (t, \cdot,\cdot)\|^2_{  L_2 (\bR^{d+d_1})} -  \| u (s, \cdot,\cdot)\|^2_{  L_2 (\bR^{d+d_1})  }\big|  \, dt.
$$
By the Lebesgue differentiation theorem, the above expression converges to $0$ as $\varepsilon \to 0$ for a.e. $s \in (-\infty, T]$.
\end{proof}

\begin{proof}[Proof of Theorem \ref{theorem 1.10}]
First, note that by Remark \ref{remark 2.2}, we only need to prove the assertion $(i)$.

$(i)$
 By pairing  both sides of Eq. \eqref{1.3.2} with $2 u$ and using Lemma \ref{lemma 7.2}, the Cauchy-Schwartz inequality, and Assumptions \ref{assumption 2.1} and  \ref{assumption 2.3},
for a.e.  $s \in (-\infty, T]$, we obtain
\begin{equation}
\begin{aligned}
			\label{eq7.3}
	&\|u (s, \cdot, \cdot)\|^2_{ L_2 (\bR^{d+d_1}) } + \delta \|D_{\sfv} u\|^2_{ L_2 (\bR^{1+d+d_1}_s) } + (\lambda - N_1) \|u\|^2_{ L_2 (\bR^{1+d+d_1}_s) } \\
& \le  N \|\vec f\|^2_{ L_2 (\bR^{1+d+d_1}_s) } + N \lambda^{-1} \|g\|^2_{ L_2 (\bR^{1+d+d_1}_s) },
\end{aligned}
\end{equation}
where $N_1 = N_1 (d, d_1, \delta , L)$
and $N = N (d, d_1, \delta)$.
Taking $\lambda \ge 2 N$, we may replace $\lambda - N$ with $\lambda/2$.
Finally, by this and the fact that  \eqref{eq7.3} holds for a.e. $s \in (-\infty, T]$, the desired estimate \eqref{eq1.10.1} is valid, which also implies the uniqueness part of the assertion $(i)$.

To prove the existence, due to the method of continuity and the a priori estimate \eqref{eq1.10.1}, we only need to prove that    $(Y  - \Delta_{\sfv} + \lambda) C^{\infty} (\bR^{1+d+d_1})$
is dense in $\bH^{-1}_2 (\bR^{1+d+d_1})$ for $\lambda > 0$.
Assume the opposite is true. Then, by duality, there exists a nonzero $u \in H^{1}_2 (\bR^{1+d+d_1})$ such that  the equality
$$
	- Y u  - \Delta_{   \sfv } u + \lambda u = 0
$$
holds in the sense of distributions.
 Mollifying the above equation with the mollifier defined in \eqref{eq7.1.1}  gives
$$
	- Y u_{\varepsilon}  - \Delta_{  \sfv } u_{\varepsilon}
+ \lambda u_{\varepsilon} =   (Y u)_{\varepsilon}-Y u_{\varepsilon}.
$$
Then,  replacing $t$ with $-t$ in the a priori estimate  \eqref{eq1.10.1} and using \eqref{eq7.1.2}, we get
$$
	\lambda^{1/2} \|u_{\varepsilon}\|_{ L_2 (\bR^{1+d+d_1}) }  \le N \varepsilon^{\theta/2} \|u\|_{ L_2 (\bR^{1+d+d_1}) }.
$$
Passing to the limit as $\varepsilon \to 0$ in the above inequality, we conclude $u \equiv 0$, which gives a contradiction. The theorem is proved.
\end{proof}

\appendix
\section{}
\begin{lemma}[Lemma A.1 in \cite{DY_21}]
			\label{lemma 4.7}
Let $\sigma > 0$,  $R > 0$, $p \ge 1$ be  numbers, and
$f \in L_{p, \text{loc} } (\bR^d)$. Denote
$$
	g ( x)
	=
	\int_{|y| > R^3 }
	f ( x+y)	|y|^{- (d + \sigma)}   \, dy.
$$
Then,
$$
   (|g|^p)^{1/p}_{  B_{R^3 }  }\leq N (d , \sigma) R^{- 3 \sigma}
	 \sum_{k = 0}^{\infty}
	 2^{- 3 k\sigma}
	(|f|^p)^{1/p}_{ B_{ (2^kR)^3 } }.
$$
\end{lemma}

\begin{lemma}
            \label{lemma A.2}
Let $s \in (0, 1/2)$. Then, the following assertions  hold.

$(i)$ One has
\begin{equation}
    \label{eqA.2.1}
        D_x (-\Delta_x)^{-s} u (x) = N (d, s) \, \text{p.v.} \int u (x-y) \frac{y}{|y|^{d-2s+2}} \, dy, \, \, u \in \mathcal{S} (\bR^d).
\end{equation}
This formula also holds for  $u \in C^1_0 (\bR^d)$ (see Definition \ref{definition 1.1}).

$(ii)$ For any  $u \in C^2_0 (\bR^d)$,
\begin{equation}
    \label{eqA.2.2}
    \big(D_x (-\Delta_x)^{-s}\big) \big((-\Delta_x)^{s} u\big) \equiv D_x u.
\end{equation}
\end{lemma}

\begin{proof}
It is well known that for any $u \in \mathcal{S} (\bR^d)$ (see, for example, Chapter 5 of \cite{S_70}),
$$
    (-\Delta_x)^{-s} u (x) = N_0 (d,s) \int u (x-y) \frac{1}{|y|^{d-2s}} \, dy.
$$
Differentiating under the integral's sign and integrating by parts, we obtain
\begin{align*}
 &     N_0^{-1} D_x  (-\Delta_x)^{-s} u (x) = \int D_x u (x-y) \frac{1}{|y|^{d-2s}} \, dy \\
&
   = - \lim_{\varepsilon \downarrow 0} \int_{|y| > \varepsilon} D_y u (x-y) \frac{1}{|y|^{d-2s}} \, dy
    = -(d-2s) \lim_{\varepsilon \downarrow 0} \int_{|y| > \varepsilon}  u (x-y) \frac{y}{|y|^{d-2s+2}} \, dy,
\end{align*}
which proves the first part of the assertion $(i)$.

Next, since $y |y|^{-d+2s-2}$ is an odd function,
we have
\begin{equation}
    \label{eqA.2.3}
\begin{aligned}
    &
    \bigg|\int_{|y| > \varepsilon}  u (x-y) \frac{y}{|y|^{d-2s+2}} \, dy\bigg|
  \le   \int_{|y| > 1} |u (x-y)| \frac{dy}{|y|^{d-2s+1}} \, dy\\
&
    + \int_{\varepsilon < |y| < 1} |u (x-y) - u (x)| \frac{y}{|y|^{d-2s+2}} \, dy
    \le N (d, s)\|u\|_{ C^1 (\bR^d) }.
\end{aligned}
\end{equation}
This bound combined with a limiting argument enables us to extend the formula \eqref{eqA.2.1} for $u \in C^1_0 (\bR^d)$.

$(ii)$ First,  for any $u \in C^{  2 }_0 (\bR^d)$, by  \eqref{eqA.2.3} and \eqref{eq1.20},
\begin{equation}
    \label{eqA.2.4}
\begin{aligned}
    &\|D_x (-\Delta_x)^{-s} \big((-\Delta_x)^{s} u\big)\|_{ L_{\infty} (\bR^d) } \\
    &\le  N (d, s) \|(-\Delta_x)^{s} u\|_{ C^1 (\bR^d) } \le N (d, s) \|u\|_{ C^{ 2} (\bR^d) },
\end{aligned}
\end{equation}
so that the left-hand side of \eqref{eqA.2.2} is well defined.
Furthermore, note that   \eqref{eqA.2.2}  holds if $u \in \mathcal{S} (\bR^d)$.
Then, the desired assertion follows from  \eqref{eqA.2.4} and a limiting argument.
\end{proof}

\begin{theorem}[Corollaries 3.2 and 3.5 of \cite{DY_21}]
            \label{theorem A.4}
Let  $c \geq 1$, $K \ge 1$, $p, q, r_1, \ldots, r_d > 1$
be  numbers,
$T \in (-\infty, \infty]$,
and $f \in L_{p, r_1, \ldots, r_d,q} (\bR^{1+2d}_T, w)$, where $w$ is  given by \eqref{eq2.2}, and $w_i, i  = 0, 1, \ldots, d$ satisfy \eqref{eq1.0}. Then, the following assertions hold.

$(i)$ (Hardy-Littlewood type theorem)
$$
	\| \bM_{c, T} f \|_{ L_{p,r_1, \ldots, r_d, q} (\bR^{1+2d}_{T}, w) }
	\leq N (d, p, q, r_1, \ldots, r_d, K) \| f \|_{ L_{p, r_1, \ldots, r_d, q} (\bR^{1+2d}_T, w) }.
$$

$(ii)$ (Fefferman-Stein type theorem)
$$
	\| f \|_{ L_{p , r_1, \ldots, r_d, q} (\bR^{1+2d}_T, w) }
	\leq N (d, p, q, r_1, \ldots, r_d, K)	
	\| f^{\#}_{c, T} \|_{  L_{p,  r_1, \ldots, r_d, q} (\bR^{1+2d}_T, w) }.
$$

$(iii)$ For $\alpha \in (-1, p-1)$, the above inequalities also hold in the space
$$
	L_{p; r_1, \ldots, r_d} (\bR^{1+2d}_T,  |x|^{\alpha} \prod_{i = 1}^d w_i (v_i))
$$
with
$
	N = N (d, p,  r_1, \ldots, r_d, K, \alpha).
$
\end{theorem}

\begin{lemma}[Lemma A.2 in \cite{DY_21}]
            \label{lemma 9.7}
Let $p > 1, K \ge 1$ be numbers,
 $w \in A_p (\bR^d)$   be such that $[w]_{A_p (\bR^d)} \le K$,
and $f \in L_p (\bR^d, w)$.
Then, there exists a number $p_0    = p_0 (d, p, K)  > 1$ such that
$f \in L_{p_0, \text{loc}} (\bR^d)$.
\end{lemma}

\begin{lemma}
            \label{lemma A.5}
Let $p \in (1, \infty)$ and $u \in L_p (\bR^d)$ be a function such that  $(-\Delta)^{1/3}u \in L_p (\bR^d)$.
Then, for any $\varepsilon > 0$,
$$
    \|(-\Delta_x)^{1/6} u \|_{ L_p (\bR^d) }
    \le  N  \varepsilon \|(-\Delta_x)^{1/3} u\|_{ L_p (\bR^d) } + N \varepsilon^{-1} \|u\|_{ L_p (\bR^d) },
$$
where $N = N (d, p)$.
\end{lemma}

\begin{proof}
It follows from the  Hormander-Mikhlin multiplier theorem that $u \in H^{1/3}_p (\bR^d)$, where the latter is the Bessel potential space (see the definition, for example, in Chapter 13 of \cite{Kr_08}).
Then, by the Hormander-Mikhlin multiplier theorem and the properties of the Bessel potential space (see, for example, \cite{Kr_08}),
\begin{align*}
      \|(-\Delta_x)^{1/6} u\|_{ L_p (\bR^d) } &\le N
     \|(1-\Delta_x)^{1/6}u\|_{ L_p (\bR^d) } \le N  \|(1-\Delta_x)^{1/3} u\|_{ L_p (\bR^d) }\\
&
    \le N \|(-\Delta_x)^{1/3} u \|_{ L_p (\bR^d) }  + N \|u\|_{ L_p (\bR^d) }.
 \end{align*}
 Now the desired assertion follows from the scaling argument.
\end{proof}

\begin{lemma}[Lemma 7.2 in \cite{DY_21}]
			\label{lemma 9.10}
Let $\gamma_0 > 0$ be a number and $R_0$ be the constant  in Assumption \ref{assumption 2.2} $(\gamma_0)$. Let $r \in (0, R_0/2)$, $c > 0$ be numbers.
 Then, one has
$$
	I: = \fint_{  Q_{r, c r}  } |a (t, x, v) - (a (t, \cdot, \cdot))_{  B_{r^3} \times B_r  }| \, dz  \le N (d) c^3 \gamma_0.
$$
\end{lemma}

\begin{lemma}
            \label{lemma A.7}
Let $p \in (1, \infty), \alpha \in (-d, d (p-1))$ be numbers, $u \in L_p (\bR^d, |x|^{\alpha})$, and $\xi$ be a  measurable function satisfying the bound
$$
   |\xi (y)| \le N_0 (1+|y|)^{-d-\beta}, y \in \bR^d,
$$
for some $\beta > 0$. Let $\xi_{\varepsilon} = \varepsilon^{-d}\xi(\cdot/\varepsilon)$.
Then, $u \ast \xi_{\varepsilon}\ \in L_p (\bR^d, |x|^{\alpha})$, and
\begin{equation}
            \label{eq A.7.1}
    \|u \ast \xi_\varepsilon\|_{L_p (\bR^d, |x|^{\alpha})} \le N (d, p, \alpha, \beta, N_0) \|u\|_{L_p (\bR^d, |x|^{\alpha})}.
\end{equation}
Furthermore, if we assume, additionally, that $\xi \in C^{\infty}_0 (\bR^d)$ is a function with the unit integral, then, $u \ast \xi_{\varepsilon} \to 0$ in $L_p (\bR^d, |x|^{\alpha})$.
\end{lemma}

\begin{proof}
Note that for any $x \in \bR^d$,
\begin{align}
& |u \ast \xi_\varepsilon (x)| \le  N \fint_{|y| < 1} |u (x-\varepsilon y)| \, dy \notag\\
&+ N  \sum_{k=0}^{\infty} 2^{-\beta k} \fint_{ 2^k < |y| < 2^{k+1} } |u(x-\varepsilon y)| \, dy
    \le     N M u (x),\label{eq12.18}
\end{align}
where $N =N (d, N_0, \beta)$,
and $M$ is the usual Hardy-Littlewood maximal function.
Since $|x|^{\alpha}, \alpha \in (-d, d (p-1))$ is an $A_p (\bR^d)$ weight (see Remark \ref{rem2.1}), \eqref{eq A.7.1} follows from a version of the Hardy-Littlewood maximal inequality in weighted Lebesgue spaces (see \cite{AM_84}).

To prove the second assertion, we note that $u \ast \xi_{\varepsilon}$ converges to $u$ as $\varepsilon \to 0$ a.e. due to Lemma \ref{lemma 9.7} and the Lebesgue differentiation theorem. Now the claim follows from \eqref{eq12.18} and the dominated convergence theorem.
\end{proof}

\end{document}